\documentclass[english,12pt]{smfart}
\usepackage{amssymb}
\usepackage{amsfonts}
\usepackage{amscd}
\usepackage[T1]{fontenc}
\usepackage[all]{xypic}
\usepackage{mathrsfs}
\usepackage{color}

\CompileMatrices

\swapnumbers

\def\ord{{\rm ord}}
\def\ac{{\overline{\rm ac}}}

\def\Var{{\rm Var}}
\def\Def{{\rm Def}}

\def\RDef{{\rm RDef}}

\def\ordjac{{\rm ordJac}}

\def\11{{\mathbf 1}}
\def\AA{{\mathbb A}}

\def\CC{{\mathbb C}}

\def\FF{{\mathbb F}}

\def\LL{{\mathbb L}}

\def\NN{{\mathbb N}}

\def\QQ{{\mathbb Q}}
\def\RR{{\mathbb R}}

\def\ZZ{{\mathbb Z}}

\def\cC{{\mathscr C}}

\def\cK{{\mathcal K}}
\def\cL{{\mathcal L}}
\def\cM{{\mathcal M}}

\def\cO{{\mathcal O}}
\def\cP{{\mathcal P}}
\def\cQ{{\mathcal Q}}

\def\cT{{\mathcal T}}

\def\llp{\mathopen{(\!(}}

\def\rrp{\mathopen{)\!)}}

\newtheorem{theorem}[subsection]{Theorem}
\newtheorem{lem}[subsection]{Lemma}
\newtheorem{cor}[subsection]{Corollary}
\newtheorem{prop}[subsection]{Proposition}

\theoremstyle{definition}
\newtheorem{definition}[subsection]{Definition}

\newtheorem{def-prop}[subsection]{Proposition-Definition}
\newtheorem{def-theorem}[subsection]{Theorem-Definition}
\newtheorem{def-lem}[subsection]{Lemma-Definition}

\theoremstyle{remark}

\theoremstyle{plain}

\numberwithin{equation}{subsection}

\def\boxit#1#2{\setbox1=\hbox{\kern#1{#2}\kern#1}%
\dimen1=\ht1 \advance\dimen1 by #1
\dimen2=\dp1 \advance\dimen2 by #1
\setbox1=\hbox{\vrule height\dimen1 depth\dimen2\box1\vrule}%
\setbox1=\vbox{\hrule\box1\hrule}%
\advance\dimen1 by .4pt \ht1=\dimen1
\advance\dimen2 by .4pt \dp1=\dimen2 \box1\relax}

\renewcommand{\theequation}{\thesubsection.\arabic{equation}}

\newcommand{\mX}{\ensuremath{\mathcal{X}}}

\def\Jac{{\rm Jac}}

\def\Vol{{\rm Vol}}

\begin{document}

\title[Motivic integration in all residue field characteristics]
{Motivic integration in all residue field characteristics for Henselian discretely valued fields of characteristic zero}

\author{Raf Cluckers}

\address{Universit\'e Lille 1, Laboratoire Painlev\'e, CNRS - UMR 8524, Cit\'e Scientifique, 59655
Villeneuve d'Ascq Cedex, France, and,
Katholieke Universiteit Leuven, Department of Mathematics,
Celestijnenlaan 200B, B-3001 Leu\-ven, Bel\-gium\\  }
\email{Raf.Cluckers@math.univ-lille1.fr}
\urladdr{http://math.univ-lille1.fr/$\sim$cluckers}

\author{Fran\c cois Loeser}

\address{Institut de Math\'ematiques de Jussieu,
UMR 7586 du CNRS,
Universit\'e Pierre et Marie Curie,
Paris, France}
\email{loeser@math.jussieu.fr}
\urladdr{http://www.math.jussieu.fr/$\sim$loeser/}


\maketitle

\renewcommand{\partname}{}

\section{Introduction}
Though one can consider
Motivic Integration to have quite satisfactory foundations in  residue characteristic zero
after \cite{CL}, \cite{exp} and  \cite{hk},
much remains to be done
in positive residue characteristic. The aim
of the present paper is to explain how one can extend the formalism and results from \cite{CL}
to mixed characteristic. Other aims are to give an axiomatic approach instead of using only the Denef-Pas language, and to extend the formalism of \cite{CL} to one with richer angular component maps.

\par
Let us start with some motivation.
Let $K$ be a fixed finite field extension of $\QQ_p$ with residue field $\FF_q$ and let $K_d$ denote its unique unramified extension of degree $d$, for $d \geq 1$.
Denote by
$\cO_d$ the ring of integers of $K_d$ and fix a polynomial
$H \in \cO_1 [x_1, \cdots, x_n]$. For each $d$ one can consider
the Igusa local zeta function
$$
Z_d (s) = \int_{\cO_d^n} \vert H (x) \vert_d^s \vert dx \vert_d,
$$
with $\vert \_ \vert_d$ and $\vert dx \vert_d$ the corresponding norm and Haar measure  such that the measure of $\cO_d$ is $1$ and such that  $|a|_d$ for any $a\in K_d$ equals the measure of $a\cO_d$.  Meuser in \cite{Meus}
proved that  there exist polynomials $G$ and $H$ in $\ZZ [T, X_1, \cdots, X_t]$ and complex numbers
$\lambda_1, \cdots, \lambda_t$ such that, for all $d \geq 1$,
$$
Z_d (s) = \frac{G( q^{-ds} , q^{d \lambda_1}, \cdots, q^{d \lambda_t})}{H( q^{-ds} , q^{d \lambda_1}, \cdots, q^{d \lambda_t})}.
$$
Later Pas \cite{Past}, \cite{Pas2} extended Meuser's result to more general integrals.
In view of \cite{jag} and \cite{JAMS}, it is thus natural to expect that there exists a motivic rational function
$Z_{\mathrm{mot}} (T)$ with coefficients in a certain Grothendieck ring
such that, for every $d \geq 1$,
$Z_d (s)$ is obtained from
$Z_{\mathrm{mot}} (T)$ by a counting procedure and by putting
$T$ equal to $q^{-ds}$.
The theory presented here allows to prove such a result (more generally for $H$ replaced by a definable function), see Proposition \ref{Meuser}.

\par
Another motivation for the present work lies in joint work with J.~Nicaise \cite{cln}, where we prove some cases
of a conjecture by Chai on  the additivity of his base change conductor for semi-abelian varieties \cite{Chai} \cite{Chai-Yu},
by using the Fubini Theorems and change of variables results for Motivic Integration in arbitrary residual characteristic.

\par
Amongst the achievements of motivic integration is the definition of measure and integrals on more general Henselian valued fields than just locally compact ones, for example on Laurent series fields over a characteristic zero field \cite{k}, \cite{arcs}, on complete discrete valuation rings with perfect residue field \cite{LSeb}, \cite{Seb1}, \cite{Ni-trace}, and on algebraically closed valued fields \cite{hk}. Another important use of motivic integration, initiated  in \cite{JAMS}, and continued in \cite{CLR}  \cite{CL} \cite{exp} \cite{CGH}, is as a tool to interpolate $p$-adic integrals for all finite field extension of $\QQ_p$ and integrals over $\FF_q\llp t\rrp $, uniformly in big primes $p$ and its powers $q$, and thus to understand their properties when $p$ varies.
This has more recently led to powerful and general Transfer Principles \cite{exp} \cite{CGH}, which, for example, allow to transfer equalities from integrals defined over $\QQ_p$ to equalities of integrals defined over $\FF_p\llp t\rrp$, and vice versa. This Transfer Principle is used to change the characteristic in statements like the Fundamental Lemma in the Langlands program, see \cite{CHL} \cite{ZYGordon}. In this paper we will expand on the first two above-mentioned directions of motivic integration: to measure on more general fields with arbitrary residue field characteristic, and to interpolate many $p$-adic integrals by a single motivic integral.
For fixed $p$ with $p=0$ or $p$ a prime number, and fixed integer $e\geq 0$, we will define the motivic measure and integrals on all Henselian discretely valued fields of characteristic $(0,p)$, and ramification degree $e$ in the case that $p\not =0$, which will correspond with the Haar measure in the case of $p$-adic fields. Our approach will be uniform in all Henselian field extensions with the same ramification degree $e$, and hence, it will give an interpolation of $p$-adic integrals for all $p$-adic fields with ramification degree $e$.
Let us note that the present work finds its roots in work by Denef \cite{D84} and \cite{D85}, who combined model theory with integration to prove a rationality conjecture by Serre.

\par

A basic tool in our approach is to use higher order angular components maps $\ac_n$ for integers $n\geq 1$, already used by Pas in \cite{Past},  where $\ac_n$ is a certain multiplicative map from the valued field $K$ to the residue ring $\cO_K/\cM_K^n$ with $\cM_K$ the maximal ideal of the valuation ring $\cO_K$. We use several structure results about definable sets and definable functions in first order languages involving the $\ac_n$, one of which is called cell decomposition and goes back to \cite{Past} and \cite{CLR}. We implement our approach with the $\ac_n$ also in equicharacteristic zero discretely valued Henselian fields. This has the advantage of providing much more
definable sets than with $\ac = \ac_1$ only, for instance  all cylinders over definable sets are definable with the $\ac_n$, which is not the case if one uses only the usual angular component $\ac$. In mixed characteristic there is a basic interplay between the residue characteristic $p$, the ramification degree $e$, and the angular component maps $\ac_n$, for example when one applies Hensel's Lemma. Therefore, in mixed characteristic, we in fact need to consider higher order residue rings in the setup, instead of only considering the residue field as in \cite{CL} \cite{exp}.

\par

Similarly as in \cite{CL} we systematically study families of motivic integrals.
We have tried to give a more direct approach to definitions and properties of the motivic measure and functions than in \cite{CL}: instead of the existence-uniqueness theorem of section 10 of \cite{CL}, we explicitly define the motivic integrals and the integrability conditions and we do this step by step, as an iteration of more simple integrals. These explicit definitions give the same motivic measure and integrals as the ones that come from a direct image framework. Of course one has to be careful when translating conditions about integrability of a nonnegative function on a product space to conditions about the iterated integral.

\par
One new feature that does not appear in \cite{CL}, and which provides more flexibility in view of future applications, is the usage of the abstract notion of $\cT$-fields, where $\cT$ stands for a first order theory. The reader has the choice to work with some of the listed more concrete examples of $\cT$-fields (which are close to the concrete semi-algebraic setup of \cite{CL} or the subanalytic setup of \cite{CLR}) or with axiomatic, abstract $\cT$-fields. Thus, $\cT$-fields allow one to work with more general theories $\cT$ than the theories in the original work by Pas. (In \cite{CL}, the set-up is restricted to the original language of Denef-Pas with its natural theory.)

\par
Similarly as in \cite{CL}, we prove a general change of variables formula, a general Fubini Theorem,  the theory may be specialized  to previously known versions of motivic integration (e.g. as in \cite{LSeb}),  interpolates
 $p$-adic integrals, no completion of any Grothendieck ring is needed, and we implement how to integrate motivically with a motivic measure associated to a volume form on an algebraic variety.
Importantly, the theory is flexible enough  to work in various parametrized set-ups where the parameters can come from the valued field, the residue field, and the value group  (this last property has been very useful in \cite{CHL} and \cite{ZYGordon}).

\par
To make our work more directly comparable and linked with \cite{CL}, we write down in Section 11 how our more direct definitions of integrable constructible motivic functions lead naturally to a direct image formalism, analogous to the one in \cite{CL}.
Let us indicate how \cite{CL} and this paper complement each other, by an example. Having an equality between two motivic integrals as in \cite{CL} implies that the analogous equality will hold over all $p$-adic fields for $p$ big enough and all fields $\FF_q\llp t\rrp$ of big enough characteristic (the lower bound can be computed but is usually very bad). This leaves one with finitely  many `small' primes $p$, say, primes which are less than $N$. For the fields $\FF_q\llp t\rrp$ of small characteristic, very little is known in general and one must embark on a case by case study. On the other hand, in mixed characteristic, one could use the framework of this paper finitely many times to obtain the equality for all $p$-adic fields with residue characteristic less than $N$ and bounded ramification degree. Note that knowing an equality for a small prime $p$ and all possible ramification degrees is more or less equivalent to knowing it in $\FF_p\llp t\rrp$, which as we mentioned can be very hard.

\par
We end Section 11 with a comparison with work by J. Sebag and the second author on motivic integration in a smooth, rigid, mixed characteristic context,
and also in the context of smooth varieties with a volume form.
These comparisons play a role in \cite{cln}.
 The results of this paper have been announced in the mixed characteristic case in \cite{clip0} without proofs. Here we give all proofs more generally in both the mixed characteristic and the equal  characteristic  $0$ case.
\par

\medskip
{\small The research leading to these results has received funding from the European Research Council under the European Community's Seventh Framework Programme (FP7/2007-2013) / ERC Grant Agreement nr. 246903 NMNAG, and from the Fund for Scientific Research - Flanders (G.0415.10).}

\section{A concrete setting}\label{sec:conc}

\subsection{}A discretely valued field $L$ is a field with a surjective group homomorphism $\ord:L^\times\to\ZZ$, satisfying the usual axioms of a non-archimedean valuation.
A ball in $L$ is by definition a set of the form
$ \{x\in L\mid \gamma \leq \ord ( x - a)\} $,
where $a\in L$ and $\gamma\in \ZZ$. The collection of balls in $L$ forms a base for the so-called valuation topology on $L$.
The valued field $L$ is called Henselian if its valuation ring $\cO_L$ is a Henselian ring. Write $\cM_L$ for the maximal ideal of $\cO_L$.

\par

In the whole paper we will work with the notion of $\cT$-fields, which is more specific than the notion of discretely valued field, but which can come with additional structure if one wants. The reader who wants to avoid the formalism of $\cT$-fields may skip Section \ref{sec:bmin} and directly go to Section \ref{sec1} and use the following concrete notion of $(0,p,e)$-fields instead of $\cT$-fields.

\begin{definition}\label{0pe}
Fix an integer $e\geq 0$ and let $p$ be either $0$ or a prime number.
A $(0,p,e)$-field is a Henselian, discretely valued field $K$
of characteristic $0$, residue field characteristic $p$, together with a chosen uniformizer $\pi_K$ of the valuation ring $\cO_K$ of $K$, such that either $0=p=e$ or $p>0$ and the ramification degree equals $e$, meaning that $\ord \pi_K^e= \ord p = e$.
\end{definition}

We will always identify the value group of a $(0,p,e)$-field with the ordered group of integers $\ZZ$. The field $\QQ_p$ together with, for example, $p$ as a uniformizer is a $(0,p,1)$-field, as well as the algebraic closure of $\QQ$ inside $\QQ_p$, or any unramified, Henselian field extension of $\QQ_p$.
A $(0,p,e)$-field $K$ comes with natural so-called higher order angular component maps for $n\geq 1$,
$$
\ac_n:K^\times \to (\cO_K\bmod \cM_K^n) : x\mapsto \pi_K^{-\ord x} x \bmod \cM_K^n
$$
extended by $\ac_n(0)=0$. Sometimes one writes  $\ac$ for $\ac_1$. Each map $\ac_n$ is multiplicative on $K$ and coincides on $\cO_K^\times$ with the natural projection $\cO_K \to \cO_K/ \cM_K^n$.

\subsection{}
To describe sets in a field independent way, we will use first order languages, where the following algebraic one is inspired by languages of Denef and Pas. Its name comes from the usage of higher order angular component maps, namely modulo positive powers of the maximal ideal.
Consider the following basic language
$\cL_{\rm high}$ which has a sort for the valued field, a sort for the value group, and a sort for each residue ring of the valuation ring modulo $\pi^n$ for integers $n>0$. On the collection of these sorts,  $\cL_{\rm high}$ consists of the language of rings for the valued field together with a symbol $\pi$ for the uniformizer, the language of rings for each of the residue rings, the Presburger language $(+,-,0,1,\leq,\{\cdot \equiv\cdot \bmod n\}_{n>1})$ for the value group, a symbol $\ord$ for the valuation map,
symbols $\ac_n$ for integers $n>0$ for the angular component maps modulo the $n$-th power of the maximal ideal, and projection maps $p_{ n, m}$ between the residue rings for $n\geq m$.
On each $(0,p,e)$-field $K$, the language $\cL_{\rm high}$ has its natural meaning, where $\pi$ stands for $\pi_K$, $\ord$ for the valuation $K^\times\to \ZZ$, $\ac_n$ for the angular component map $K\to \cO_K/\cM_K^n$, and  $p_{n,m}$ for the natural projection map from $\cO_K/ \cM_K^n$ to $\cO_K / \cM_K^m$.

\par
 Let $\cT_{(0,p,e)}$ be the theory in the language $\cL_{\rm high}$ of sentences that are true in all $(0,p,e)$-fields. Thus, in particular, each $(0,p,e)$-field is a model of $\cT_{(0,p,e)}$.
 In this concrete setting, we let $\cT$ be $\cT_{(0,p,e)}$ in the language $\cL_{\rm high}$, and $\cT$-field means $(0,p,e)$-field.
See \cite{Past} for a concrete list of axioms that imply the whole theory $\cT_{(0,p,e)}$.

\section{Theories on $(0,p,e)$-fields}\label{sec:bmin}

In total we give three approaches to $\cT$-fields in this paper, so that the reader can choose which one fits him best. The first one is the concrete setting of Section \ref{sec:conc}; the second one consists of a list of more general and more adaptable settings in Section \ref{exs}, and the third approach is the axiomatic approach for theories and languages on $(0,p,e)$-fields in Section \ref{sec:axT}.
Recall that for the first approach one takes $\cT=\cT_{(0,p,e)}$ in the language $\cL_{\rm high}$, and $\cT$-field just means $(0,p,e)$-field.

\subsection{A list of theories}\label{exs}

In our second approach, we give a list of theories and corresponding languages which can be used throughout the whole paper.

\begin{enumerate}

\item\label{sas} Most closely related to the notion of $(0,p,e)$-fields is that of $(0,p,e)$-fields over a given ring $R_0$, for example a ring of integers, using the language $\cL_{\rm high}(R_0)$. Namely, for $R_0$ a subring of a $(0,p,e)$-field, let $\cL_{\rm high}(R_0)$ be the language $\cL_{\rm high}$ with coefficients (also called parameters) from $R_0$, and let $\cT_{(0,p,e)}(R_0)$ be the theory of $(0,p,e)$-fields over $R_0$ in the language $\cL_{\rm high}(R_0)$. In this case one takes $\cT=\cT_{(0,p,e)}(R_0)$ with language $\cL_{\rm high}(R_0)$. By a $(0,p,e)$-field $K$ over $R_0$ we mean in particular that the order and angular component maps on $K$ extend  the order and angular component maps on $R_0$.

\item\label{as} In order to include analytic functions, let $K$ be a $(0,p,e)$-field, and for each integer $n\geq 1$ let $K\{ x_1,\ldots,x_n\}$ be the ring of those formal power series $\sum_{i\in\NN^n}a_ix^i$ over $K$ such that $\ord(a_i)$ goes to $+\infty$ whenever $i_1+\ldots+i_n$ goes to $+\infty$. Let $\cL_{K}$ be the language $\cL_{\rm high}$ together with function symbols for all the elements of the rings $K\{ x_1,\ldots,x_n\}$, for all $n>0$. Each complete $(0,p,e)$-field $L$ over $K$ allows a natural interpretation of the language $\cL_{K}$, where $f$ in $K\{ x_1,\ldots,x_n\}$ is interpreted naturally as a function from $\cO_L^n$ to $L$. Let $\cT_{K}$ be the theory in the language $\cL_{K}$ of the collection of complete $(0,p,e)$-fields $L$ over $K$. In this case one takes $\cT=\cT_{K}$ with language $\cL_{K}$. For an explicit list of axioms that implies $\cT_{K}$, see \cite{CLR}.

\item More generally than in the previous example, any of the analytic structures of \cite{CLip} can be used for the language with corresponding theory $\cT$, provided that $\cT$ has at least one $(0,p,e)$-field as model.

\item\label{exts} For $\cT_0$ and $\cL_0$ as in any of the previous three items let $\cT$ and $\cL$ be any expansion of $\cT_0$ and $\cL_0$, which enriches $\cT_0$ and $\cL_0$ exclusively on the residue rings sorts. Suppose that $\cT$ has at least one model which is a $(0,p,e)$-field.

\end{enumerate}

\subsection{The axiomatic set-up}\label{sec:axT}

Our third approach to $\cT$-fields consists of a list of axioms which should be fulfilled by an otherwise unspecified theory $\cT$ in some language $\cL$. The pairs of theories and languages for $(0,p,e)$-fields in the prior two approaches are examples of this axiomatic set-up by Proposition \ref{cdI} (see Proposition \ref{cTR} for more examples).

\par
In this third approach, we start with a language $\cL$ which contains $\cL_{\rm high}$ and has the same sorts as $\cL_{\rm high}$, and a theory $\cT$ which contains $\cT_{(0,p,e)}$ and which is formulated in the language $\cL$.
The sort for the valued field is called the main sort, and each of the other sorts (namely the residue ring sorts and the value group sort) are called auxiliary. It is important that no extra sorts are created along the way.

\par
The list of axioms will be about all models of $\cT$, and not only about $(0,p,e)$-fields.
Note that
any model $\cK$ of the theory $\cT_{(0,p,e)}$ with valued field $K$ carries an interpretation of all the symbols of $\cL_{\rm high}$ with the usual first order properties, even when $K$ is not a $(0,p,e)$-field\footnote{This  happens, for example,  when the value group of $K$ is not isomorphic to $\mathbb{Z}$.}. We will use the notation $\pi_K$, $\ac_n$ and so on for the meaning of the symbols $\pi$ and $\ac_n$ of $\cL_{\rm high}$, as well as the notion of balls, and so on, for all models of $\cT_{(0,p,e)}$.
The axioms below will involve parameters, which together with typical model theoretic compactness arguments will yield all the family-versions of the results we will need for motivic integration. To see in detail how such axioms are exploited, we refer to \cite{CLb}. 
By definable, resp.~$A$-definable, we will mean $\cL$-definable without parameters, resp.~$\cL$-definable allowing parameters from $A$, unless otherwise stated.

\par
The following two Jacobian properties treat close-to-linear (local) behavior of definable functions in one variable.

\begin{definition}[Jacobian property for a function]\label{jacf}
Let $K$ be the valued field of a model of $\cT_{(0,p,e)}$.
Let $F:B\to B'$ be a function with $B,B'\subset K$.
 We say that $F$ has the Jacobian property if the following conditions hold all together:
\begin{itemize}
\item[(i)] $F$ is a bijection and $B,B'$ are balls in $K$,
\item[(ii)] $F$ is $C^1$ on $B$ with derivative  $F'$,
\item[(iii)]
$$
F' \mbox{ is nonvanishing and }
\ord (F')  \mbox{ is constant on
$B$},
$$
\item[(iv)] for all $x,y\in B$ with $x\not=y$, one has
$$ \ord(F')+\ord(x-y)
=\ord(F(x)-F(y)).$$
\end{itemize}
If moreover $n>0$ is an integer, we say that $F$ has the $n$-Jacobian property if also the following hold
\begin{itemize}
\item[(v)]  $ \ac_n(F')$ is constant on $B$,
\item[(vi)] for all $x,y\in B$ one has
$$ \ac_n(F')\cdot \ac_n(x-y)
=\ac_n(F(x)-F(y)).$$
\end{itemize}
\end{definition}


\begin{definition}[Jacobian property for $\cT$]
Say that the Jacobian property holds for the $\cL$-theory $\cT$ if for any model $\cK$ with Henselian valued field $K$ the following holds.

For any finite set $A$ in $\cK$ (serving as parameters in whichever sorts), any integer $n>0$, and any $A$-definable function $F:K\to K$ there exists an $A$-definable function
$$
f:K\to S
$$
with $S$ a Cartesian product of (the $\cK$-universes of) sorts not involving $K$ (these are also called auxiliary sorts),
such that each infinite fiber $f^{-1}(s)$ is a ball on which $F$ is either constant or has the $n$-Jacobian property (as in Definition \ref{jacf}).
\end{definition}

The following notion of $\cT$ being split is related to the model-theoretic notions of orthogonality and stable embeddedness.

\begin{definition}[Split]
Call $\cT$ split if the following conditions hold for any model $\cK$ with Henselian valued field $K$, value group $\Gamma$ and residue rings $\cO_K/\cM_K^n$

\begin{enumerate}
\item[(i)]Any $\cK$-definable subset of $\Gamma^r$ is $\Gamma$-definable in the language $(+,-,0,<)$.
\item[(ii)]For any finite set $A$ in $\cK$, any  $A$-definable subset $X\subset (\prod_{i=1}^s\cO_K/\cM_K^{m_i}) \times \Gamma^r$ is equal to  a finite disjoint union of $Y_i\times Z_i$ where the $Y_i$ are $A$-definable subsets of $\prod_{i=1}^s\cO_K/\cM_K^{m_i}$, and the $Z_i$ are $A$-definable subsets of $\Gamma^r$.
\end{enumerate}
\end{definition}

The general notion of $b$-minimality is introduced in \cite{CLb}. Here we work with a version which is more concretely adapted to the valued field setting.

\begin{definition}[Finite $b$-minimality]\label{fb}
Call $\cT$ finitely $b$-minimal if the following hold for any model $\cK$ with Henselian valued field $K$. Each locally constant $\cK$-definable function $g:K^\times\to K$ has finite image, and,
for any finite set $A$ in $\cK$ (serving as parameters in whichever sorts) and any $A$-definable set $X\subset K$, there exist an integer $n>0$, an $A$-definable function
$$
f:X\to S
$$
with $S$ a Cartesian product of (the $\cK$-universes of) sorts not involving $K$ (also called auxiliary sorts),
and an $A$-definable function
$$
c:S\to K
$$
such that each nonempty fiber $f^{-1}(s)$ for $s\in S$ is either
\begin{enumerate}
\item equal to the singleton $\{c(s)\}$, or,
\item\label{fb2} equal to the ball $\{x\in K\mid \ac_n(x-c(s))=\xi(s),\ \ord (x-c(s)) = z(s)\}$ for some $\xi(s)$ in $\cO_K/\cM_K^n$ and some $z(s)\in \Gamma$.
\end{enumerate}
\end{definition}
Note that in the above definition, the values $z(s)$ and $\xi(s)$ depend uniquely on $s$ in the case that $f^{-1}(s)$ is a ball and can trivially be extended when $f^{-1}(s)$ is not a ball so that  $s\mapsto z(s)$ and $s\mapsto \xi(s)$ can both be seen as $A$-definable functions on $S$.

\begin{lem}\label{finiteb}
For any model $\cK$ with valued field $K$ of a finitely $b$-minimal theory, any definable function from a Cartesian product of (the $\cK$-universes of) auxiliary sorts to $K$ has finite image, and so does any definable, locally constant function from any definable set $X\subset K^n$ to $K$.
\end{lem}
\begin{proof}
Take a model $\cK$ with valued field $K$ of the finitely $b$-minimal theory.
Suppose that $h$ is a $\cK$-definable function from an auxiliary set $S$ (that is, $S$ is a definable subset of a Cartesian product of the $\cK$-universes of auxiliary sorts) to the valued field $K$, and that the image of $h$ is infinite. We have to show a contradiction.

\par
Suppose that $S$ is a definable subset of the product $S_1\times\ldots\times S_n$ of universes of auxiliary sorts. If $n=1$ and if $S_n$ is the value group, then let $g_1:K^\times \to K$ be the function sending $x$ to $h(\ord(x))$ if this is well-defined and to $0$ otherwise. In this case $g_1$ is locally constant and has infinite image, a contradiction to finite $b$-minimality.
In the case that $n=1$ and $S_n$ is a residue ring $\cO_K/\pi^\ell_K$ for some $\ell\geq 1$, there exists $\xi\in \cO_K/\pi^\ell_K$ and $\ell_0\geq 0$ such that $S\cap \xi+\pi_K^{\ell_0}(\cO_K/\cM^\ell_K)^\times$ is mapped to an infinite set under $h$. In this case the function $g_2:K^\times \to K$ which maps  $x$ to $h(\xi+\pi_K^{\ell_0}\cdot \ac_{\ell}(x))$ if this is well-defined and to $0$ otherwise is locally constant and has infinite image, again a contradiction to finite $b$-minimality.

\par
For $n>1$, we may suppose by induction on $n$ that the coordinate projection $p:S\mapsto \prod_{i<n}S_i$ has fibers which are mapped to finite sets under $h$, that is, $h(p^{-1}(s))$ is finite for any $s\in \prod_{i<n}S_i$. By model theoretic compactness, we may as well suppose that the nonempty fibers of $p$ are mapped under $h$ to sets with exactly $t$ elements. Let $g_0,\ldots,g_{t-1}$ be the elementary symmetric polynomials in $t$ variables. Write $h_\ell: \prod_{i<n}S_i \to K, s\mapsto h_\ell(s)$ for the evaluation of $g_\ell$ on the $t$-element set $h(p^{-1}(s))$ if nonempty, and to zero if this set is empty. At least one of the functions $h_\ell$ for $\ell=0,\ldots,t - 1$ must have infinite image, and we are done by induction on $n$.

\par
For the next statement, suppose that $h:X\subset K^n\to K$ is definable, locally constant, and has infinite image. For $n=1$, by finite $b$-minimality, there exist $\cK$-definable
$$
f:X\to S
$$
with $S$ a Cartesian product of (the $\cK$-universes of) auxiliary sorts,
and a $\cK$-definable function
$$
c:S\to K
$$
such that each nonempty fiber $f^{-1}(s)$ for $s\in S$ is either
\begin{enumerate}
\item equal to the singleton $\{c(s)\}$, or,
\item equal to the ball $\{x\in K\mid \ac_n(x-c(s))=\xi(s),\ \ord (x-c(s)) = z\}$ for some $\xi(s)$ in $\cO_K/\cM_K^n$  and some $z\in \ZZ$,
\end{enumerate}
By what we have just proven, the image of $c$ is finite. Up to a finite partition of $X$, we may suppose that $c$ is a constant function, and that each nonempty fiber $f^{-1}(s)$ for $s\in S$ is equal to the ball $\{x\in K\mid \ac_n(x-c(s))=\xi(s),\ \ord (x-c(s)) = z\}$. After a translation we may suppose that $c=0$ on $S$. Now the extension of $h$ to a function $K^\times \to K$ by zero is locally constant. We are done for $n=1$ by finite $b$-minimality.
For general $n$, let $p:K^n\to K^{n-1}$ be a coordinate projection. By the case $n=1$ one has that $h(y,\cdot)$ sending $t$ such that $(y,t)$ lies in $X$ to $h(y,t)$ has finite image. By model theoretic compactness and up to a finite partition of $X$, we may suppose, for each $y\in p(X)$, that $h(y,\cdot)$ has $t$ elements. Now one finishes by induction using the elementary symmetric polynomials as before in the proof.
\end{proof}

\begin{cor}\label{b-min}
A finitely $b$-minimal theory is in particular $b$-minimal (as defined in \cite{CLb}).
\end{cor}
\begin{proof}
Immediate from Lemma \ref{finiteb}: axiom (b1) of Definition 2.1 of \cite{CLb} is contained in our definition of finite $b$-minimality; axiom (b2) of loc.~cit.~follows from the first statement of Lemma \ref{finiteb}; finally, axiom (b3) of loc.~cit.~follows from the second statement of Lemma \ref{finiteb} and Criterium 2.6 ($*$) of \cite{CLb}.
\end{proof}

Finally we come to the most general notion of $\cT$-fields, namely the axiomatic one of our third approach.

\begin{definition}\label{def:T}
Let $\cT$ be a theory containing $\cT_{(0,p,e)}$ in a language $\cL$ with the same sorts as $\cL_{\rm high}$, which is split, finitely $b$-minimal, has the Jacobian property, and has at least one $(0,p,e)$-field as model.
Then by a $\cT$-field we mean a $(0,p,e)$-field which is a model of $\cT$.
\end{definition}

We have  the following variant of the cell decomposition statement and related structure results on definable sets and functions of \cite{CLip} for our more concrete theories.
\begin{theorem}[\cite{CLip}]\label{cdI}
The theory $\cT_{(0,p,e)}$ as well as the listed theories in \ref{exs} satisfy the conditions of Definition \ref{def:T}.
\end{theorem}

Finally we indicate how one can create new theories with properties as in Definition \ref{def:T}.
\begin{prop}\label{cTR}
Let $\cT$ be a theory that satisfies the conditions of Definition \ref{def:T}. Then so does the theory $\cT(R)$ in the language $\cL(R)$ for any ring $R$ which is a subring of a $\cT$-field, where $\cT(R)$ is the theory of all $\cT$-fields which are algebras over $R$ (and which extend $\ord$ and the $\ac_n$ on $R$).
\end{prop}
\begin{proof}
The same argument is used to show all the desired properties.
We will make this argument explicit by showing that $\cT(R)$ is split.
Let  $A$ be a finite set in $\cK$, and  $X\subset (\prod_{i=1}^s\cO_K/\cM_K^{m_i}) \times \Gamma^r$ a $A$-definable subset in the language $\cL(R)$. In particular, only finitely many constants from $R$ play a role in the formula describing $X$, hence there exists a finite set $A'$ consisting of $A$ and a finite subset of $R$ such that $X$ is $A'$-definable in $\cL$. Now, since $\cT$ is split, $X$  equals a finite disjoint union of $Y_i\times Z_i$ where the $Y_i$ are $A'$-definable subsets of $\prod_{i=1}^s\cO_K/\cM_K^{m_i}$, and the $Z_i$ are $A'$-definable subsets of $\Gamma^r$, all in the language $\cL$. Clearly the $Y_i$ and $Z_i$ are $A$-definable in $\cL(R)$ as desired.
\end{proof}

From now on we fix $p\geq 0$ and $e\geq 0$ and one of the notions of $\cT$, $\cL$, and $\cT$-fields as in Definition \ref{def:T} for the rest of the paper, which includes the possibility of $\cT$ and $\cL$ being as in Sections \ref{sec:conc}, \ref{exs}, or as in Proposition \ref{cTR}. We will often write $K$ for a $\cT$-field instead of writing the pair $K,\pi_K$ where $\pi_K$ is a uniformizer of $\cO_K$.

\section{Definable subassignments and definable morphisms}\label{sec1}\label{fstar}

\subsection{}\label{defsub}We recall that definable means $\cL$-definable without parameters\footnote{Note that parameters from, for example, a base ring can be used, see Section \ref{exs} and Proposition \ref{cTR}.}.
For any integers $n,r,s\geq 0$ and for any tuple $m=(m_1,\ldots,m_s)$ of nonnegative integers, denote by $h[n,m,r]$ the functor sending a  $\cT$-field $K$ to
$$
h[n,m,r](K) := K^n \times (\cO_K/\cM_K^{m_1}) \times \cdots \times (\cO_K/\cM_K^{m_s}) \times \ZZ^r.
$$

The data of a subset $X_K$ of $h[n,m,r](K)$ for each  $\cT$-field $K$ is called a definable subassignment (in model theory sometimes loosely called a definable set), if there exists an $\cL$-formula $\varphi$ in tuples of free variables of the corresponding lengths and in the corresponding sorts such that $X_K$ equals $\varphi(K)$, the set of the points in $h[n,m,r](K)$ satisfying $\varphi$. If one wants to specify the theory, one writes definable $\cT$-subassignment instead of definable subassignment.

An example of a definable subassignment of $h[1,0,0]$ is the data of the subset $P_2(K)\subset K$ consisting of the nonzero squares in $K$ for each  $\cT$-field $K$, which can be described by the formula $\exists y (y^2=x\wedge x\not=0)$ in one free variable $x$ and one bound variable $y$, both running over the valued field\footnote{Note that, as is standard, to determine $\varphi(K)$, each variable occurring in $\varphi$ (thus also the variables which are bound by a quantifier and hence not free), runs over exactly one set out of $K$, $\ZZ$, or a residue ring $\cO_K/\cM_K^\ell$.}.

A definable morphism $f:X\to Y$ between definable subassignments $X$ and $Y$ is given by a definable subassignment $G$ such that $G(K)$ is the graph of a function $X(K)\to Y(K)$ for any $\cT$-field $K$. We usually write $f$ for the definable morphism, ${\rm{Graph}}(f)$ for $G$, and $f_K$ for the function $X(K)\to Y(K)$ with graph $G(K)$. A definable isomorphism is by definition a definable morphism which has an inverse.

Denote by $\Def$ (or $\Def(\cT)$ in full) the category of definable subassignments with the definable morphisms as morphisms. More generally, for $Z$ a definable subassignment, denote by $\Def_Z$ the category of definable subassignments $X$ with a specified definable morphism $X\to Z$ to $Z$, with as morphisms between $X$ and $Y$ the definable morphisms which make commutative diagrams with the specified $X\to Z$ and $Y\to Z$. We will often use the notation $X_{/Z}$ for $X$ in $\Def_Z$. In the prior publications \cite{CL} and \cite{exp}, we used the notation $X\to Z$ instead of the shorter
$X_{/Z}$.

For every morphism $f  : Z  \rightarrow Z'$ in $\Def$,
composition with $f$ defines a functor $f_! : \Def_Z \rightarrow
\Def_{Z'}$, sending $X_{/Z}$ to $X_{/Z'}$. Also, fiber product defines a functor $f^{*}:
\Def_{Z'}  \rightarrow \Def_{Z}$,
namely, by sending $Y_{/  Z'}$ to $(Y\otimes_{Z'} Z)_{/ Z}$, where for each $\cT$-field $K$ the set $(Y\otimes_{Z'} Z )(K)$ is the set-theoretical fiber product of $Y(K)$ with $Z(K)$ over $Z'(K)$ with the projection as specified function to $Z(K)$.

Let $Y$ and $Y'$ be in $\Def$. We write $Y
\times Y'$ for the subassignment corresponding to the Cartesian product and we write $Y [n, m, r]$ for
$Y \times h [n,m, r]$. (We fix in the whole paper $h = h [0, 0, 0]$ to be the definable subassignment of the singleton $\{0\}$, that is, $h(K)=\{0\}=K^0$ for all $K$, so that $h[n,m,r]$, as previously defined, is compatible with the notation of $Y[n,m,r]$  for general $Y$.)

\par
By a point on a definable subassignment $X$ we mean a tuple $x=(x_0,K)$ where $K$ is a  $\cT$-field and $x_0$ lies in $X(K)$. We denote $|X|$ for the collection of all points that lie on $X$.

Let us give a result which is true for all these examples, by quantifier elimination results (which are absent in the axiomatic setting of Section \ref{sec:axT}).

\begin{lem}\label{compl}
Let $\cT$ and $\cL$ be as in any of the examples 1 -- 4 of Section \ref{exs}. Let $X$ and $Y$ be definable subassignments of $h[m,n,r]$ for some $m,n,r$. If for each $\cT$-field $L$ which is complete for the valuation topology one has $X(L)=Y(L)$, then $X=Y$ as definable subassignments.
\end{lem}
\begin{proof}
By the elimination results for valued field quantifiers in certain definitial expansions of the language corresponding to $\cT$ of \cite{CLip}, it follows that, for any given $\cT$-field $K$, whether $X(K)=Y(K)$ or not only depends on the isomorphism classes of the residue rings $\cO_K/\cM_K^n$ of $K$. Also, for any $\cT$-field $K$, its completion is also a $\cT$-field (in particular, the completion is an $\cL$-structure). This is clear in the semi-algebraic examples (that is, without non-algebraic analytic functions), and in the subanalytic case that $\cL$ contains non-algebraic analytic functions, this follows from Definition 4.1.6 of \cite{CLip}.   Since the completion of $K$ and $K$ itself have isomorphic such residue rings, the lemma follows.
\end{proof}

\subsection{Dimension}

Since $\cT$ is in particular $b$-minimal in the sense of \cite{CLb} by Corollary \ref{b-min}, for each $\cT$-field $K$ and each definable subassignment $\varphi$ we can take the dimension of $\varphi(K)$ to be as defined in \cite{CLb}, and use the dimension theory from \cite{CLb}. In the context of finite $b$-minimality, for nonempty and definable $X\subset h[n,m,r](K)$, this dimension is defined by induction on $n$, where for $n=0$ the dimension of $X$ is defined to be zero, and, for $n=1$, $\dim X=1$ if and only if $p(X)$ contains a ball where $p:h[1,m,r](K)\to K$ is the coordinate projection, and one has $\dim X=0$ otherwise. For general $n\geq 1$, the dimension of such $X$ is the maximal number $r>0$ such that for some coordinate projection  $p:h[n,m,r](K)\to K^r$, $p(X)$ contains a Cartesian product of $r$ balls if such $r$ exists and the dimension is $0$ otherwise. Note that a nonempty definable $X\subset h[n,m,r](K)$ has dimension zero if and only if it is a finite set.

The dimension of a definable subassignment $\varphi$ itself is defined as the maximum of all $\varphi(K)$ when $K$ runs over all  $\cT$-fields.

For $f:X\to Y$ a definable morphism and $K$ a $\cT$-field, the relative dimension of the set $X( K)$ over $Y( K)$ (of course along $f_K$) is the maximum of the dimensions of the fibers of $f_K$, and the relative dimension of the definable assignment $X$ over $Y$ (along $f$) is the maximum of these over all $K$.

One has all the properties of \cite{CLb} for the dimensions of the  sets $\varphi(K)$ and the related properties for the definable subassignments themselves, analogous to the properties of the so-called $K$-dimension of \cite{CL}.

\section{Summation over the value group}\label{piint}
We consider a formal symbol $\LL$ and the ring
$$\AA := \ZZ \Bigl[\LL, \LL^{- 1},
\Bigl(\frac{1}{1 - \LL^{- i}}\Bigr)_{i >0}\Bigr],
$$
as subring of the ring of rational functions in $\LL$ over $\QQ$.
Furthermore, for each real number $q>1$, we consider the ring morphism
$$
\theta_q:\AA\to\RR:r(\LL)\mapsto r(q),
$$
that is, one evaluates the rational function $r(\LL)$  in $\LL$ at $q$.

Recall that $h[0,0,1]$ can be identified with $\ZZ$, since $h[0,0,1](K)=\ZZ$ for all $\cT$-fields $K$. Let $S$ be in $\Def$, that is, let $S$ be a definable subassignment.  A definable morphism $\alpha: S\to h[0,0,1]$ gives rise to a function (also denoted by $\alpha$) from $|S|$ to $\ZZ$ which sends a point $(s,K)$ on $S$ to $\alpha_K(s)$. Likewise, such $\alpha$ gives rise to the function $\LL^\alpha$ from $|S|$ to $\AA$ which sends a point $(s,K)$ on $S$ to $\LL^{\alpha_K(s)}$.

We define the ring $\cP
(S) $ of constructible Presburger functions on $S$ as the subring of
the ring of functions $|S| \rightarrow \AA$ generated by
\begin{enumerate}
\item all constant functions into $\AA$,
\item all functions $\alpha: |S| \rightarrow \ZZ$ with $\alpha:S\to h[0,0,1]$ a definable morphism,
\item all functions of the form $\LL^{\beta}$ with $\beta:S\to h[0,0,1]$ a
definable morphism.
\end{enumerate}

Note that a general element of $\cP
(S) $ is thus a finite sum of terms of the form
$ a \LL^\beta \prod_{i=1}^\ell \alpha_i$
with $a\in \AA$, and the $\beta$ and $\alpha_i$ definable morphisms from $S$ to $h[0,0,1]=\ZZ$.

For any $\cT$-field $K$,
any $q > 1$ in $\RR$, and $f$ in $\cP(S)$ we write
$\theta_{q,K}(f)  : S(K) \rightarrow \RR$ for the function sending $s\in S(K)$ to $\theta_q(f(s,K))$.

\par
Define a
partial ordering on $\cP(S)$ by setting $f\geq 0$ if for  every
 $q > 1$ in $\RR$ and every $s$ in $|S|$, $\theta_{q} (f(s)) \geq 0$. We denote by
$\cP(S)_+$ the set $\{f \in \cP(S) \, | \, f \geq 0\}$.
Write $f \geq g$ if $f -  g$ is in $\cP_+ (S) $. Similarly, write $\AA_+$ for the sub-semi-ring of $\AA$ consisting of the non-negative elements of $\AA$, namely those elements $a$ with $\theta_{q}(a)\geq 0$ for all real $q > 1$.

\par

Recall the notion of summable families in $\RR$ or $\CC$,
cf.~\cite{bbk} VII.16. In particular, a family $(z_i)_{i \in I}$ of
complex numbers is summable if and only if the family $(|z_i|)_{i
\in I}$ is summable in $\RR$.

\par
We shall say a
function $\varphi$ in $\cP (h[0,0,r])$ is integrable if for each $\cT$-field $K$ and for each real $q>1$, the family $(\theta_{q,K}(\varphi) (i))_{i \in \ZZ^r}$ is
summable.

\par
More generally we shall say a
function $\varphi$ in $\cP (S[0,0,r])$ is $S$-integrable if for each $\cT$-field $K$, for each real $q>1$, and for each $s\in S(K)$,  the family $(\theta_{q,K}(\varphi) (s, i))_{i \in \ZZ^r}$ is
summable.
The latter notion of $S$-integrability is key to all integrability notions in this paper.

\par
We denote by ${\rm I}_S\cP (S[0,0,r])$ the
collection of $S$-integrable functions in $\cP (S[0,0,r])$. Likewise, we denote by ${\rm I}_{S}\cP_+ (S[0,0,r])$ the
collection of $S$-integrable functions in $\cP_+ (S[0,0,r])$. Note that ${\rm I}_S\cP (S[0,0,r])$, resp.~${\rm I}_{S}\cP_+ (S[0,0,r])$,  is a $\cP (S)$-module, resp.~a $\cP_+(S)$-semi-module.

\par
The following is inspired by results in \cite{D85} and appears in a similar form in the context of \cite{CL}; the proof of Theorem-Definition 4.5.1 of \cite{CL} and the arguments of Section 4.6 of \cite{CL} go through. This uses finite $b$-minimality, the fact that $\cT$ is split, basic results about Presburger sets and functions, and explicit calculations for geometric series and their derivatives.

\begin{def-theorem}\label{thm:kjh}
For each $\varphi$ in ${\rm I}_{S}\cP (S[0,0,r])$ there exists
a unique function $\psi=\mu_{/S}(\varphi)$
in $\cP (S)$ such that for all
$q
> 1$, all $\cT$-fields $K$, and all $s$ in $S(K)$
\begin{equation}\label{kjh}
\theta_{q,K} (\psi) (s) = \sum_{i \in \ZZ^r} \theta_{q,K}
(\varphi) (s, i).
\end{equation}
Moreover, the mapping $\varphi\mapsto \mu_{/S}(\varphi)$ yields a
morphism of $\cP (S)$-modules
$$
\mu_{/S} : {\rm I}_S \cP (S \times
\ZZ^r) \longrightarrow \cP (S).
$$
\end{def-theorem}

Clearly, the above map $\mu_{/S}$ sends ${\rm I}_S \cP_+ (S \times
\ZZ^r)$ to $\cP_+ (S)$.
For $Y$ a definable subassignment of $S$, we denote by $\11_Y$
the function in $\cP (S)$ with value 1 on $Y$ and zero on $S
\setminus Y$. We shall denote by $\cP^0 (S)$ (resp.~$\cP^0_+ (S)$)
the subring (resp.~sub-semi-ring) of $\cP (S)$ (resp.~$\cP_+ (S)$)
generated by the functions $\11_Y$ for all definable
subassignments $Y$ of $S$ and by the constant function $\LL - 1$.

\par
If $f : Z
\rightarrow Y$ is a morphism in $\Def$, composition with $f$
yields natural pullback morphisms $f^* : \cP (Y) \rightarrow \cP (Z)$ and
$f^* : \cP_+ (Y) \rightarrow \cP_+ (Z)$. These pullback morphisms and the subrings $\cP^0 (S)$ will play a role for the richer class of motivic constructible functions. First we turn our attention to another ingredient for motivic constructible functions, coming from the residue rings. Afterwards we will glue these two ingredients together along the common subrings $\cP^0_+ (S)$ to define motivic constructible functions.

\section{Integration over the residue rings}\label{rdbis}

\subsection{}\label{Q+}On the integers side we have defined rings of (nonnegative) constructible Presburger functions $\cP_+(\cdot)$ and a summation procedure of these functions over subsets of $\ZZ^r$. On the residue rings side we will proceed differently.

\par

Let $Z$ be a definable subassignment in $\Def$.
Define the semi-group
$ \cQ_+(Z)$
as the quotient of
the free abelian semi-group
over symbols $[Y]$ with $Y_{/Z}
$ a subassignment of $Z[0,m,0]$ for some $m=(m_1,\ldots,m_s)$ with $m_i\geq 0$ and $s\geq 0$, with as distinguished map from $Y$ to $Z$ the natural projection, by the following relations.
\begin{itemize}

\item[]
\begin{equation}\label{eq0}
[\emptyset] = 0, \mbox{ where $\emptyset$ is the empty subassignment.}
\end{equation}

\item[]
\begin{equation}\label{eq1}
[Y] = [Y']
\end{equation}
if there exists a definable isomorphism $Y\to Y'$ which commutes with the projections $Y\to Z$ and $Y'\to Z$.

\item[]
\begin{equation}\label{eq2}
[Y_1 \cup Y_2] + [Y_1 \cap Y_2]
= [Y_1] + [Y_2]
\end{equation}
for $Y_1$ and $Y_2$ definable subassignments of a common $Z[0,m,0]$ for some $m$.

\item[]\begin{equation}\label{eq3}
[Y] = [Y']
\end{equation}
if for some definable subassignment $W$ of $Z[0,m,0]$ with $m = (m_1, \ldots, m_s)$, one has
$Y' = p^{-1}(W)$ and $Y=W[0,1,0]$ with $p:Z[0,(m_1+1 , m_2 \ldots, m_s),0]\to
Z[0,m,0]$ the projection.

\end{itemize}

\par
We will still write $[Y]$ for the class of $[Y]$ in $ \cQ_+ ( Z ) $ for $Y\subset Z[0,m,0]$. The relations (\ref{eq1}) and (\ref{eq3}) force the classes of $Z[0,m,0]$ and $Z[0,m',0]$ in $ \cQ_+(Z)$ to be identified for any tuples $(m_i)_i$ and $(m'_j)_j$ satisfying $\sum_j m'_j = \sum_i m_i$.
If the theory $\cT$ imposes the residue field to be perfect and if $p>0$, then the relations (\ref{eq3}) are redundant, which can be seen as follows. If $p>0$, then the map $x\mapsto x^p$ on the valuation ring induces a definable injective morphism $j$ from $h[0,1,0]$ into $h[0,2,0]$. Indeed, for any $x$ in $\cO_K$ and any $m$  in $\cM_K$, the elements $x^p$ and $(x+m)^p$ are congruent modulo $\cM_K^2$. Now, if the theory $\cT$ imposes the residue field to be perfect and still $p>0$, then $h[0,(1,1),0]\to h[0,2,0]:(x,y)\mapsto j(x)+t\cdot j(y)$ is a definable isomorphism, where $t$ here is an abbreviation for $\ac_2(1+\pi)-\ac_2(1)$, and one can do similarly for the relations (\ref{eq3}) in general. Note that perfectness of a field $k$ of characteristic $p>0$ ensures that $x\mapsto x^p$ is a bijection.
In \cite{CL}, the longer notation $SK_0(\RDef_Z)$ is used instead of $\cQ_+(Z)$ and the relations (\ref{eq3}) do not occur since only $\ac_n$ with $n=1$ is used in \cite{CL}.

\subsection{}\label{semiring} The semi-group $\cQ_+(Z)$ carries a semi-ring structure with multiplication for $Y\subset Z[0,m,0]$ and $Y'\subset Z[0,m',0]$ given by
$$
[Y]\cdot [Y'] := [Y\otimes_Z Y'],
$$
where the fibre product is taken along the coordinate projections to $Z$.
Similarly, for $f : Z_1
\rightarrow Z_2$ any morphism in $\Def$, there is a natural pullback homomorphism of semi-rings $f^* : \cQ_+ (Z_2) \rightarrow \cQ_+ (Z_1)$ which sends $[Y]$ for some $Y\subset Z_2[0,m,0]$ to $[Y\otimes_{Z_2} Z_1]$.
Write $\LL$ for the class of $Z[0,1,0]$ in $ \cQ_+ ( Z ) $. Then, by relations (\ref{eq3}) and (\ref{eq1}), one has that the class of $Z[0,m,0]$ in  $ \cQ_+ ( Z ) $ equals $\LL^{|m|}$ with $m=(m_i)_i$ and $|m|=\sum_i m_i$.
Clearly, for each $a\in \cQ_+(Z)$, there exists a tuple  $m$ and a $Y\subset Z[0,m,0]$ such that $a=[Y]$.

\par
To preserve a maximum of information at the level of the residue rings, we will integrate functions in $\cQ_+(\cdot)$ over residue ring variables in a formal way.
Suppose that $Z = X[0,k,0]$ for some tuple $k$, let $a$ be in $\cQ_+(Z)$ and write $a$ as $[Y]$ for some $Y\subset Z[0,n,0] $. We write $\mu_{/X}$ for the corresponding formal integral in the fibers of the coordinate projection $Z\to X$
$$
\mu_{/X}: \cQ_+(Z) \to \cQ_+(X): [Y] \to [Y],
$$
where the class of $Y$ is first taken in $\cQ_+(Z)$ and then in $\cQ_+(X)$.
Note that this allows one to integrate functions from $\cQ_+$ over residue ring variables, but of course not over valued field neither over value group variables. To integrate over any kind of variables, we will need to combine the value group part $\cP_+$ and the residue rings part $\cQ_+$.

\section{Putting $\cP_+$ and $\cQ_+$ together to form $\cC_+$}\label{C+}

\subsection{}
Many interesting functions on Henselian valued fields have a component that comes essentially from the value group and one that comes from residue rings.
For $Z$ in $\Def$, we will glue the pieces $\cP_+(Z)$ and $\cQ_+(Z)$ together by means of the common sub-semi-ring $\cP_+^0(Z)$.  Recall that $\cP_+^0(Z)$ is the sub-semi-ring of $\cP_+(Z)$ generated by the characteristic functions $\11_Y$ for all definable
subassignments $Y\subset Z$ and by the constant function $\LL - 1$.
\par

Using the canonical semi-ring  morphism  $\cP^0_+ (Z) \rightarrow \cQ_+(Z)$, sending $\11_Y$ to $[Y]$ and $\LL- 1$ to $\LL - 1$, we define the semi-ring $\cC_+(Z)$  as
$$
\cP_+(Z)\otimes_{\cP_+^0(Z)}\cQ_+(Z).
$$
We call elements of $\cC_+(Z)$ (nonnegative) constructible motivic functions on $Z$.

\par
If $f : Z
\rightarrow Y$ is a morphism in $\Def$, we find natural pullback morphisms $f^* : \cC_+ (Y) \rightarrow \cC_+ (Z)$, by the tensor product definition of $\cC_+(\cdot)$. Namely, $f^*$ maps $\sum_{i=1}^r a_i\otimes b_i$ to $\sum_i f^*(a_i)\otimes f^*(b_i)$, where $a_i\in \cP_+ (Y)$ and $b_i\in \cQ_+ (Y)$.

\subsection{Evaluation at points}\label{eval}
For a definable subassignment $X$ and a function $\varphi$ in $\cC_+ (X)$, one can evaluate $\varphi$ at points lying on $X$ as follows. For any $\cT$-field $K$, one may consider the theory $\cT(K)$ in the  language $\cL(K)$ as in Proposition \ref{cTR}. For any $\cT(K)$-subassignment $Z$, let us temporarily write $\cQ_{+,K}(Z)$ for the object $\cQ_+(Z)$ as defined in Section \ref{Q+} with $\cT(K)$ and $\cL(K)$ instead of $\cT$ and $\cL$. Let us write likewise $\cP_{+,K}(Z)$, $\cC_{+,K}(Z)$, and so on, when $\cT(K)$ and $\cL(K)$ are used instead of $\cT$ and $\cL$.
Let $X$ be a definable subassignment and take $\varphi$ in $\cC_+ (X)$. For any $\cT$-field $K$ and element $x_0$ of $X(K)$, let us write $X_x$ for the definable $\cT(K)$-subassignment such that, for any $\cT(K)$-field $L$, the set $X_x(L)$ is the singleton $\{x_0\}$ if $x_0$ lies on $X(L)$, and the empty set otherwise. The evaluation of $\varphi$ at the point $x=(x_0,K)$ of $X$ is denoted by $i_x^*(\varphi)$, and is defined as the element of $\cC_{+,K}(X_x)$ given by fixing, inside all the formulas involved in the description of $\varphi$, the tuple of variables running over $X$ by the tuple $x_0$. Other notions can be defined similarly. For example, for a definable morphism $f:Y\to X$, a point $(x_0,K)$ on $X$, and $X_x$ as above in this section, one defines the fiber $f^{-1}(X_x)$ as the definable $\cT(K)$-subassignment given by the conjunction of the formula describing $Y$ with a $\cL(K)$-formula expressing that $f(y)=x_0$.

\subsection{Interpretation in non-archimedean local fields.}\label{nonarch}

An important feature of our setting (as well as in the settings of \cite{CL}, \cite{exp}, and \cite{JAMS}) is that the motivic constructible functions and their integrals interpolate actual functions and their integrals on non-archimedean local fields, and even more generally on $\cT$-fields with finite residue field.

\par

Let $X\subset h[n,m,r]$ be in $\Def$,  let $\varphi$ be in $\cC_+(X)$, and let $K$ be a $\cT$-field with finite residue field. In this case $\varphi$ gives rise to an actual set-theoretic function $\varphi_K$ from $X(K)$ to $\QQ_{\geq 0}$, defined as follows:

\par
For $a$ in  $\cP_+(X)$, one gets $a_K:X(K)\to \QQ_{\geq 0}$ by replacing $\LL$ by $q_K$, the number of elements in the residue field of $K$.

\par
For $b=[Y]$ with $Y$ a subassignment of $X[0,m,0]$ in $\cQ_+(X)$, if one writes $p:Y(K)\to X(K)$ for the projection, one defines $b_K:X(K)\to \QQ_{\geq 0}$ by sending $x\in X(K)$ to $\# (p^{-1}(x))$, that is, the number of points in $Y(K)$ that lie above $x\in X(K)$.

\par
For our general $\varphi$ in $\cC_+(X)$, write $\varphi$ as a finite sum $\sum_i a_i\otimes b_i$ with $a_i\in \cP_+(X)$ and $b_i\in \cQ_+(X)$. Our general definitions are such that the function
\begin{equation*}
\varphi_K :
 X(K)\to \QQ_{\geq 0} :
 x\mapsto \sum_i a_{iK}(x)\cdot b_{iK}(x)
\end{equation*}
does not depend on the choices made for $a_i$ and $b_i$.

\subsection{Integration over residue rings and value group}

We have the following form of independence (or orthogonality) between the integer part and the residue rings part of $\cC_+(\cdot)$.

\begin{prop}\label{peprod0}
Let $S$ be in $\Def$.
The canonical morphism
$$
\cP_+ ( S[0,0,r]) \otimes_{\cP^0_+ (S)} \cQ_+(S[0,m,0])
\longrightarrow \cC_+ (S[0,m,r])
$$
is an isomorphism of semi-rings, where the homomorphisms $p^*:\cP^0_+ (S)\to \cP_+ (S[0,0,r])$ and $q^*:\cP^0_+ (S)\to \cQ_+(S[0,m,0])$ are the pullback homomorphisms of the projections $p:S[0,0,r]\to S$ and $q:S[0,m,0]\to S$.
\end{prop}

The mentioned canonical morphism of Proposition \ref{peprod0} sends $a\otimes b$ to $p_1^*(a)\otimes p_2^*(b)$, where $p_1:S[0,m,r]\to S[0,0,r]$ and $p_1:S[0,m,r]\to S[0,m,0]$ are the projections.
\begin{proof}Direct consequence of the fact that $\cT$ is split.
\end{proof}
Recall that for $a$ in $\cQ_+(X)$, one can write $a=[Y]$ for some $Y$ in $\Def_X$, say, with specified morphism $f:Y\to X$. We shall write $\11_a := \11_{f(Y)}$ for the characteristic function of $f(Y)$, the ``support'' of $a$.

\begin{def-lem}\label{def:aux}
Let $\varphi$ be in $\cC_+(Z)$ and suppose that $Z=X[0,m,r]$ for some $X$ in $\Def$. Say that $\varphi$ is $X$-integrable if one can write $\varphi = \sum_{i=1}^\ell a_i\otimes b_i$ with $a_i\in \cP_+(X[0,0,r])$ and $b_i\in \cQ_+(X[0,m,0])$ as in Proposition \ref{peprod0} such that moreover the $a_i$ lie in ${\rm I}_{X}\cP_+ (X[0,0,r])$ in the sense of Section \ref{piint}.
If this is the case, then
$$
\mu_{/X}(\varphi):= \sum_i \mu_{/X}(a_i) \otimes \mu_{/X}(b_i)\ \in \cC_+(X)
$$
does not depend on the choice of the $a_i$ and $b_i$ and is called the integral of $\varphi$ in the fibers of the coordinate projection $Z\to X$.
\end{def-lem}
\begin{proof}
Using the natural maps also occurring in Proposition \ref{peprod0}, the Lemma-Definition can be restated that the map $w$ on the free abelian semi-group $W$ on ${\rm I}_{X}\cP_+ (X[0,0,r])\times \cQ_+(X[0,m,0])$ sending $\sum_i (a_i,b_i)$ to $\sum_i \mu_{/X}(a_i) \otimes \mu_{/X}(b_i)$ factorizes through the tensor product of semi-groups ${\rm I}_{X}\cP_+ (X[0,0,r])\otimes_{\cP^0_+ (S)} \cQ_+(X[0,m,0])$.
But this follows from the obvious linearity properties of $w$, namely, that $cw(a,b)=w(ca,b)=w(a,cb)$, $w(a+a',b)=w(a,b)+w(a',b)$ and $w(a,b+b')=w(a,b)+w(a,b')$  for $(a,b)$ and $(a',b')$ in $W$ and $c\in \cP^0_+ (S)$.
\end{proof}

The following lemma is a basic form of a projection formula which concerns pulling a factor out of the integral if the factor depends on other variables than the ones that one integrates over.

\begin{lem}
\label{projaux}
Let $\varphi$ be in $\cC_+(Z)$ such that $\varphi$ is $X$-integrable, where $Z=X[0,m,r]$ for some $X$ in $\Def$. Let $\psi$ be in $\cC_+(X)$ and let $p:Z\to X$ be the projection. Then $p^*(\psi)\varphi$ is $X$-integrable and
$$
\mu_{/X}(p^*(\psi)\varphi) = \psi \mu_{/X}(\varphi)
$$
holds in $\cC_+(X)$.
\end{lem}

Note that Lemma \ref{projaux} is immediate when $m=0$.
\begin{proof}
By Lemma-Definition \ref{def:aux}, the definition of $p^*$, and the linearity of $\mu_{/X}$ on $\cP_+$ and on $\cQ_+$.
\end{proof}
Using the natural morphisms $\cP_+(Z)\to \cC_+(Z)$ which sends $\psi$ to $\psi\otimes [Z]$, and $\cQ_+(Z)\to \cC_+(Z)$ which sends $\nu$ to $\11_Z\otimes \nu$, we can formulate the following. (Note that $\cP_+(Z)\to \cC_+(Z):b\mapsto \11_Z\otimes b$ is not necessarily injective neither necessarily surjective.)

\begin{lem}\label{lift}
For any $\varphi\in \cC_+(Z)$ there exist $\psi$ in $\cP_+(Z[0,m,0])$ and $\nu $ in $\cQ_+(Z[0,0,r])$ for some $m$ and $r$ such that $\nu$ is $Z$-integrable and $\varphi = \mu_{/Z}(\psi) = \mu_{/Z}(\nu )$.
\end{lem}
\begin{proof}
Clear by the fact that $\cT$ is split.
\end{proof}

Here is a first instance of the feature that  relates integration of motivic functions with actual integration (or summation) on $\cT$-fields with finite residue field.
\begin{lem}\label{nonarchsum}
Let $\varphi$ be in $\cC_+(Z)$ and suppose that $Z=X[0,m,r]$ for some $X$ in $\Def$. Let $K$ be a $\cT$-field with finite residue field and consider $\varphi_K$ as in Section \ref{nonarch}.  If $\varphi$ is $X$-integrable then, for each $x\in X(K)$, $\varphi_K(x,\cdot):y\mapsto \varphi_K(x,y)$ is integrable against the counting measure, and  if one writes $\psi$ for $\mu_{/X}(\varphi)$, then
$$
\psi_K (x) = \sum_{y}  \varphi_K(x,y)
$$
for each $x\in X$, where the summation is over those $y$ such that $(x,y)\in Z(K)$.
\end{lem}
\begin{proof}
Clear by the definitions of $\varphi_K$ and $\mu_X$.
\end{proof}

\section{Integration over one valued field variable}

For the moment let $K$ be any discretely valued field. For a ball $B\subset K$ and for any real number $q>1$, define $\theta_{q}(B)$ as the real number $q^{-\ord b}$, where $b\in K^\times$ is such that $B=a+b\cO_K$ for some $a\in K$. We call $\theta_{q}(B)$ the $q$-volume of $B$.

Next we will define a naive and simple notion of step-function. Finite $b$-minimality will allow us to reduce part of the integration procedure to step-functions. A finite or countable collection of balls in $K$, each with different $q$-volume, is called a step-domain. We will identify a step-domain $S$ with the union of the balls in $S$. This is harmless since one can recover the individual balls from their union since they all have different $q$-volume. Call a nonnegative real valued function $\varphi:K\to\RR_{\geq 0}$ a step-function if there exists a unique step-domain $S$ such that $\varphi$ is constant and nonzero on each ball of $S$ and zero outside $S\cup \{a\}$ for some $a\in K$. Note that requiring uniqueness of the step-domain $S$ for $\varphi$ is redundant, except when the residue field has two elements.

\par

Let $q>1$ be a real number.  Say that a step-function $\varphi:K\to\RR_{\geq 0}$ with step-domain $S$ is $q$-integrable over $K$ if and only if
\begin{equation}\label{q-in}
\sum_{B\in S}  \theta_{q}(B)\cdot  \varphi (B)<\infty,
\end{equation}
where one sums over the balls $B$ in $S$, and then the expression (\ref{q-in}) is called the $q$-integral of $\varphi$ over $K$.
Using Theorem \ref{thm:kjh} one proves the following.
\begin{def-lem}\label{step-int} Suppose that $Z=X[1,0,0]$ for some $X$ in $\Def$.
Let $\varphi$ be in $\cP_+(Z)$. Call $\varphi$ an $X$-integrable family of step-functions if for each $\cT$-field $K$, for each $x\in X(K)$, and for each $q>1$, the function
\begin{equation}\label{thetaqphi0}
\theta_{q,K}(\varphi)(x,\cdot):K\to\RR_{\geq 0}:t \mapsto \theta_{q,K}(\varphi)(x,t)
\end{equation}
is a step-function which is $q$-integrable over $K$. If $\varphi $ is such a family, then there exists a unique function $\psi$ in $\cP_+(X)$ such that
$\theta_{q,K}(\psi)(x)$ equals the $q$-integral over $K$ of (\ref{thetaqphi0}) for each $\cT$-field $K$, each $x\in X(K)$, and each $q>1$.
We then call $\varphi$ $X$-integrable, we write
$$
\mu_{/X}(\varphi):=\psi
$$
and call $\mu_{/X}(\varphi)$ the integral of $\varphi$ in the fibers of $Z\to X$.
\end{def-lem}
\begin{proof}
Direct consequence of Theorem \ref{thm:kjh}. Indeed, for all $K$, $x\in X(K)$, and $t\in K$, the value of $\theta_{q,K}(\varphi(x,t))$ only depends on the $q$-volume (and thus of the radius) of the unique ball in the step-domain of $\theta_{q,K}(\varphi)(x,\cdot)$ containing $t$ if there is such ball and this value is zero if there is no such ball, hence it is clear how to replace $\varphi$ by some $\varphi_0$ in $\cP_+(X[0,0,1])$ such that one can take $\psi=\mu_{/X}(\varphi_0)$, the latter being defined in \ref{thm:kjh}. Uniqueness of $\psi$ with the desired properties is clear by the definition of $\cP_+(\cdot)$.
\end{proof}

Finally we define how to integrate a general motivic constructible function over one valued field variable, in families.

\begin{def-lem}\label{integrableC}
Let $\varphi$ be in $\cC_+(Z)$ and suppose that $Z=X[1,0,0]$. Say that $\varphi$  is $X$-integrable if there exists $\psi$ in $\cP_+(Z[0,m,0])$ with $\mu_{/Z}(\psi)=\varphi$ as in Lemma \ref{lift} such that $\psi$ is $X[0,m,0]$-integrable in the sense of Lemma-Definition \ref{step-int} and then
$$
\mu_{/X}(\varphi):= \mu_{/X} (\mu_{{/X}[0,m,0]} (\psi))\ \in \cC_+(X)
$$
is independent of the choices and is called the integral of $\varphi$ in the fibers of $Z\to X$.
\end{def-lem}
The proof of \ref{integrableC} is similar to the proofs in section 9 of \cite{CL}. We give a detailed outline for the convenience of the reader.
\begin{proof}

For any alternative $\psi'$ in $\cP_+(Z[0,m',0])$, there exists another alternative $\psi''$ in $\cP_+(Z[0,m'',0])$ such that moreover $\psi''$ is a common refinement of $\psi$ and $\psi'$ meaning that $p_!(\psi'')=\psi$ and $p_!'(\psi'')=\psi'$ where $p:Z[0,m'',0]\to Z[0,m,0]$ and $p':Z[0,m'',0]\to Z[0,m',0]$ are coordinate projections.

Hence, it is enough to consider the case that $\psi'$ is a refinement of $\psi$ in the sense that $p_!(\psi')=\psi$ with $p:Z[0,m',0]\to Z[0,m,0]$ the projection and $m'\geq m$, and to compare
$\mu_{/X} (\mu_{{/X}[0,m,0]} (\psi))$ with $\mu_{/X} (\mu_{{/X}[0,m',0]} (\psi'))$. Hence, we may moreover suppose that $m=0$ and show that
\begin{equation}
\mu_{/X} (\psi) = \mu_{/X} (\mu_{/X[0,m',0]} (\psi')),
\end{equation}
where the left hand side is as in Lemma-Definition \ref{step-int}.
 Replacing $X$ by $X[0,0,1]$ and adapting the data correspondingly, we may suppose that $\theta_{q,K}(\psi)(x,\cdot)$ is constant on a single ball $B_{x}$ and zero outside $B_x$ for each $q>1$, $K$, and $x\in X(K)$, where $B_x$ depends definably on $x$.
By finite $b$-minimality (and compactness)
we may suppose that there are an integer $N>0$ and a definable morphism $c:X[0,m',0]\to h[1,0,0]$
such that each ball in the collection of balls of $\theta_{q,K}(\psi')(x',\cdot)$ is either of the form
\begin{equation}\label{des1}
\{t\in K\mid \ord(t-c(x'))\geq z\}
\end{equation}
or
of the form
\begin{equation}\label{des2}
\{t\in K\mid \ac_n(t-c(x'))=\xi,\ \ord (t-c(x')) = z\}
\end{equation}
for some $z\in \ZZ$, $\xi\in (\cO_K/\cM_K^n)^\times$, and some $n<N$.
By Lemma \ref{finiteb}, the image $I_K(x)$ of $c_K(x,\cdot):h[0,m',0](K)\to K$ is a finite subset of $K$ for each $\cT$-field $K$ and each $x\in X(K)$, and is even uniformly bounded in size when $K$ varies (by compactness).
We may suppose that $I_K(x)$ has precisely $k$ elements, for some $k>1$ which is independent of $K$ and $x$.  We use induction on $k$.
If $k=1$, then there are two cases: either one is done by a geometric power series calculation or by Relations (\ref{eq1}) and (\ref{eq3}) of section \ref{rdbis}, see Examples 9.1.4 and 9.1.9 of \cite{CL}. Next consider $k>1$. By a geometric power series calculation as in Example 9.1.9 of \cite{CL}, we may suppose that, for $z$ as in (\ref{des2}), one has $z\leq \alpha(x)$ for some definable morphism $\alpha:X\to\ZZ$.
For each element $d$ of $I_K(x)$, let $d'$ be the average of $d$ and the elements different from $d$ that lie closest to $d$. Write $c'$ for the definable morphism that takes the values $d'$ instead of $d$. Recall that $e$ stands for the ramification degree of the $(0,p,e)$-fields we consider. Now we can change the description of the balls of $\theta_{q,K}(\psi')(x',\cdot)$ using $c'$ instead of $c$ and $n<N +N'$  as in (\ref{des1}) and (\ref{des2}), which is possible for $N'$ big enough, where big enough depends on $k$, $p$, and $e$ only. We are done by induction on $k$.
\end{proof}

\section{General integration}\label{genint}

In this section we define the motivic measure and the motivic integral of motivic constructible functions in general. For uniformity results and for applications it is important that we do this in families, namely, in the fibers of projections $X[n,m,r]\to X$ for $X$ in $\Def$.
 We define the integrals in the fibers of a general coordinate projection $X[n,m,r]\to X$ by induction on $n\geq 0$.
\begin{def-lem}\label{generalint}
Let $\varphi$ be in $\cC_+(Z)$ and suppose that $Z=X[n,m,r]$ for some $X$ in $\Def$. Say that $\varphi$  is $X$-integrable if there exist a definable subassignment $Z'\subset Z$ whose complement in $Z$ has relative dimension $<n$ over $X$, and an ordering of the coordinates on $X[n,m,r]$ such that $\varphi':=\11_{Z'}\varphi$ is $X[n-1,m,r]$-integrable and $\mu_{/X[n-1,m,r]}(\varphi')$ is $X$-integrable. If this holds then
$$
\mu_{{/X}}(\varphi):= \mu_{/X} (\mu_{/X[n-1,m,r]} (\varphi'))\ \in \cC_+(X)
$$
does not depend on the choices and is called the integral of $\varphi$ in the fibers of $Z\to X$, and is compatible with the definitions made in \ref{integrableC}.

\par
More generally, let $\varphi$ be in $\cC_+(Z)$ and suppose that $Z\subset X[n,m,r]$. Say that $\varphi$ is $X$-integrable if the extension by zero of $\varphi$ to a function $\widetilde\varphi$ in $\cC_+(X[n,m,r])$ is $X$-integrable, and define $\mu_{{/X}}(\varphi)$ as $\mu_{{/X}}(\widetilde\varphi)$. If $X$ is $h[0,0,0]$  (which is a final object in $\Def$), then we write $\mu$ instead of $\mu_{/X}$, we say integrable instead of $X$-integrable, and $\mu(\varphi)$ is called the integral of $\varphi$ over $Z$.
\end{def-lem}
 One can prove \ref{generalint} in two ways (both relying on the properties of $\cT$-fields of Definition \ref{def:T}): using more recent insights from \cite{ccl} to reverse the order of the coordinates, or, using the approach from \cite{CL} with a calculation on bi-cells. We follow the slightly shorter approach from \cite{ccl}.
\begin{proof}
If $n\leq 1$ there is nothing to prove. We proceed by induction on $n$. By permuting the coordinates if necessary, it is sufficient to prove the case that $n=2$. We may suppose that $m=r=0$. Write $p_1$ and $p_2$ for consecutive coordinate projections $p_1: X[2,0,0] \to X[1,0,0]$ and $p_2: X[1,0,0] \to X$. We may suppose that there exists $Z'$ whose complement in $Z$ has dimension $<2$ and such that $\varphi':=\11_{Z'} \varphi$ is  $X[1,0,0]$-integrable (for $p_1$) and that $\mu_{/X[1,0,0]} (\varphi)$ is $X$-integrable (for $p_2$). Up to replacing $Z$ by $Z'$, we may suppose that $\varphi=\varphi'$ and $Z=Z'$. By replacing $X$ by some $X[0,m',r']$ and by Lemma \ref{lift}, we may suppose that $\varphi$ is the image of $\psi\in \cP_+(Z)$ under the natural map $\cP_+(Z)\to \cC_+(Z):a\mapsto a\otimes [Z]$. Moreover, by finite $b$-minimality and compactness, and again by replacing $X$ by some $X[0,m',r']$, we may suppose that above each point $x$ in $X(K)$ for each $K$,
 $$
\psi_K(x,\cdot,\cdot):Z_x(K)\to \AA:(t_1,t_2)\mapsto \psi_K(x,t_1,t_2)
 $$
  is constant, and that $Z_x(K)$ has the form
 $$
 \{(t_1,t_2)\in K^2\mid t_1\in B_x,\ t_2\in B_{x,t_1}\},
 $$
 where $B_x$ is a ball only depending on $x$, and $B_{x,t_1}$ is a ball of the form
 \begin{equation}\label{ball1}
 \{t_2\in K\mid \ac_n(t_2-c(x,t_1))=\xi,\ \ord (t_2-c(x,t_1)) = z\}
\end{equation}
for some $z\in \ZZ$, $\xi \in (\cO_K/\cM_K^n)^\times$, and some $n<N$. This way, we have pushed the integrability issue into a summation problem, which would be symmetric (and thus easy) if the role of $t_1$ and $t_2$ were symmetric. We will finish the proof by reversing (piecewise) the role of the coordinates $t_1$ and $t_2$, similarly as in \cite{ccl}, for which the desiderata are easy to check.
We refer to \cite{ccl} for full details.
By the Jacobian property we may suppose that $c(x,\cdot):t_1\mapsto c(x,t_1)$ has the Jacobian property on each ball $B_x$.  In a first case we may suppose that $c(x,\cdot)$ is constant. This case being symmetric in $t_1$ and $t_2$, we are done. In the second case we suppose that the image $C_x$ of $c(x,\cdot)$ is a ball which does not contain the ball $B_{x,t_1}$. By applying finite $b$-minimality to the graph of $c$, one reduces to the first case by using the newfound center instead of $c$ to rewrite (\ref{ball1}). In the third and final case we have that the image $C_x$ of $c(x,\cdot)$ is a ball which contains $B_{x,t_1}$. Taking the inverse of $c(x,\cdot)$ on $C_x$, we can reverse the order of $t_1$ and $t_2$ and we are done by a calculation using the chain rule for derivatives. 
\end{proof}

\par
One of the main features is a natural relation between motivic integrability and motivic integration on the one hand, and classical measure theoretic integrability and integration on local fields on the other hand:

\begin{prop}\label{intmotK}
Let $\varphi$ be in $\cC_+(X[n,m,r])$ for some $X$ in $\Def$.
If $\varphi$ is $X$-integrable, then, for each local field $K$ which is a  $\cT$-field and for each $x\in X(K)$ one has that $\varphi_K(x,\cdot)$ is integrable (in the standard measure-theoretic sense). If one further writes $\psi$ for $\mu_{/X}(\varphi)$, then, for each $x\in X(K)$,
$$
\psi_K (x) = \int_{y}  \varphi_K(x,y),
$$
where the integral is against the product measure of the Haar measure on $K$ with the counting measure on $\ZZ$ and on the residue rings for $y$ running over $h[n,m,r](K)$, and where the Haar measure gives $\cO_K$ measure one.
\end{prop}
\begin{proof}
This follows from the matching of $q$ and the $\theta_q$-notions with $q_K$, the number of elements in the residue field of the local field $K$, see Section \ref{nonarch}.
\end{proof}

\subsection*{}
As an application of our framework let us explain more precisely the statement about $Z_{\mathrm{mot}} (T)$ alluded to in the introduction.
Let $K$ be a finite field extension of $\QQ_p$ with residue field $\FF_q$ and ramification degree $e$, and let $K_d$ denote its unique unramified extension of degree $d$, for $d \geq 1$. Denote by
$\cO_d$ the ring of integers of $K_d$.
We work with the language $\cL = \cL_{\rm high}(K)$ and the theory $\cT=\cT_{(0,p,e)}(K)$ as in example \ref{sas} of Section \ref{exs}.
Write $O$ for the definable subassignment of $h[1,0,0]$ given by the condition $\ord (x)\geq 0\vee x=0$. Thus, for any $\cT$-field $L$, one has $O(L)=\cO_L$.
Let $H$ be a definable morphism from $O^n$ to $O$. For example, $H$ could be given by a polynomial in $\cO_1 [x_1, \cdots, x_n]$.
For each $d\geq 1$, we consider
$$
Z_d (s) = \int_{\cO_d^n} \vert H_{K_d} (x) \vert_d^s \vert dx \vert_d,
$$
where $H_{K_d}$ is the interpretation of $H$ in $K_d$ as in Section \ref{defsub}. Igusa and Denef showed that knowing the local zeta function
$Z_d (s)$
is equivalent to knowing the series
$$
Z'_{d} (T) := \sum_{{i\in\NN} } \Vol_d(\{x\in \cO_d^n\mid  \ord (H(x) ) = i  )   T^i,
$$
by giving an explicit formula transforming $Z_d$ in $Z_d'$, cf. \cite{D84}, where $\vert dx \vert_d$ and $\Vol_d$ both stand for the Haar measure on $K_d^n$ giving $\cO_d^n$ measure $1$.
For a motivic analogue, we let, for each $i\geq 0$,  $X_i$ be the definable subassignment of $O^n$ given by the condition $\ord (H(x)) = i$.
One of the natural objects which can be compared to Igusa's local zeta functions $Z_d(s)$ mentioned in the introduction, is given by the series
$$
Z_{\mathrm{mot}} (T) := \sum_{\in{i\in\NN} } \mu(X_i)  T^i.
$$
Similarly as in Theorems 4.4.1 and 5.7.1 of \cite{CL}, $Z_{\mathrm{mot}} (T)$
is rational with a denominator which is a finite product of factors of the form $1-\LL^\alpha T^\beta$ for integers $\beta>0$ and $\alpha$.
Write $\ast$ for the point $h[0,0,0]$.
Note that $\cC(\ast)$ is nothing else than $\cQ(\ast)$ with $\LL$ and the $1-\LL^i$ for $i<0$ inverted.
For each integer $d \geq 1$ and each definable subset $A$ of $h[0,k,0]$, where $k=(k_\ell)_\ell$ is a finite tuple of non-negative integers, the number of elements on $A(K_d)$ is finite. Indeed, $A(K_d)$ is a subset of $\prod_\ell \cO_d/\cM_d^{k_\ell}$, where $\cM_d$ is the maximal ideal of $\cO_d$.
For each integer $d \geq 1$  there is a unique ring morphism $N_d :  \cC(\ast) \to \QQ$
sending the class $[A]$ of a definable subset $A$ of $h[0,k,0]$, where $k$ is a tuple, to the number of elements of $A(K_d)$.

With this notation, the following comparison result between the tower of local zeta functions $Z'_d(T)$ for $d\geq 1$ and its motivic counterpart $Z_{\mathrm{mot}} (T)$ generalize results of \cite{Meus} and \cite{Past}.

\begin{prop}\label{Meuser}
For every $d \geq 1$ one has
$$
Z'_d (T) = N_d  (Z_{\mathrm{mot}} (T)),
$$
where $N_d (Z_{\mathrm{mot}} (T))$ is obtained by evaluating $N_d$ on the coefficients of the numerator and denominator of $Z_{\mathrm{mot}}(T)$.
\end{prop}
\begin{proof}
Immediate from Proposition \ref{intmotK}.
\end{proof}

\section{Further properties}\label{fp}

As mentioned before, the projection formula allows one to pull a factor out of the integral if that factor depends on other variables than the ones that one integrates over.
\begin{prop}[Projection formula]\label{proj}
Let $\varphi$ be in $\cC_+(Z)$ for some $Z\subset X[n,m,r]$ and some $X$ in $\Def$.
Suppose that $\varphi$ is $X$-integrable, let $\psi$ be in $\cC_+(X)$ and let $p:Z\to X$ be the projection. Then $p^*(\psi)\varphi$ is $X$-integrable and
$$
\mu_{/X}(p^*(\psi)\varphi) = \psi \mu_{/X}(\varphi)
$$
holds in $\cC_+(X)$.
\end{prop}

In other words, if one would write ${\rm I}_X\cC_+(Z)$ for the $X$-integrable functions in $\cC_+(Z)$, then
$$
\mu_{/X}:{\rm I}_X\cC_+(Z) \to \cC_+(Z): \varphi \mapsto \mu_{/X}(\varphi)
$$
is a morphism of $\cC_+(X)$-semi-modules, where the semi-module structure on ${\rm I}_X \cC_+(Z)$ comes from the homomorphism $p^*: \cC_+(X) \to \cC_+(Z)$ of semi-rings, with $p:Z\to X$ the projection.

\begin{proof}[Proof of Proposition \ref{proj}]
Clear by the explicit definitions of $X$-integrals and Lemma \ref{projaux}.
\end{proof}

We will now fix our terminology concerning Jacobians and relative Jacobians, first in a general, set-theoretic setting, and then for definable morphisms.
%


For any function $h:A\subset K^n\to K^n$ (in the set-theoretic sense of function) for some $\cT$-field $K$ and integer $n> 0$,  let $\Jac h:A\to K$ be the determinant of the Jacobian matrix of $h$ where this matrix is well-defined (on the interior of $A$) and let $\Jac h$ take the value $0$ elsewhere in $A$.

In the relative case, consider a function $f:A\subset C\times K^n \to C\times K^n $ which makes a commutative diagram with the projections to $C$, with $K$ a $\cT$-field and with some set $C$. Write $\Jac_{/C} f:A\to K$ for the function satisfying for each $c\in C$ that $(\Jac_{/C}f)(c,z) = \Jac (f_c)(z) $ for each $c\in C$ and each $z\in K^n$ with $(c,z)\in A$, and where $f_c:A_c\to K^n$ is the function sending $z$ to $t$ with $f(c,z) = (c,t)$ and $(c,z)\in A$.

The existence of the relative Jacobian $\Jac g_{/X}$ in the following definable context is clear by the definability of the partial derivatives and piecewise continuity properties of definable functions.
\begin{def-lem}
Consider a definable morphism $g:A\subset X[n,0,0]\to X[n,0,0]$ over $X$ for some definable subassignment $X$. By $\Jac_{/X} g$ denote the unique definable morphism $A \to h[1,0,0]$ satisfying for each $\cT$-field $K$ that $(\Jac_{/X} g)_K= \Jac_{/X_K} (g_K)$ and call it the relative Jacobian of $g$ over $X$.
\end{def-lem}

We can now formulate the change of variables formula, in a relative setting.

\begin{theorem}[Change of variables]\label{cov}
Let $F:Z \subset X[n,0,0]  \to Z'\subset X[n,0,0]$ be a definable isomorphism over $X$ for some $X$ in $\Def$ and let $\varphi$ be in $\cC_+(Z)$.
Then there exists a definable subassignment $Y\subset Z$ whose complement in $Z$ has dimension $<n$ over $X$, and such that the relative Jacobian $ \Jac_{/X} F $ of $F$ over $X$ is nonvanishing on $Y$. Moreover, if we take the unique $\varphi'$ in $\cC_+(Z')$ with  $F^*(\varphi') = \varphi$, then $\varphi\LL^{-\ordjac_{/X} F}$ is $X$-integrable if and only if $\varphi'$ is $X$-integrable, and then
$$
\mu_{/X}(\varphi\LL^{-\ordjac_{/X} F}) = \mu_{/X}(\varphi')
$$
in $\cC_+(X)$, with the convention that $\LL^{- \ord (0)}=0$.
\end{theorem}

\begin{proof}
For $n=1$ this follows from the Jacobian property. Piecewise, the case of $n=1$ can be used to write $F$ (piecewise) as a finite composition of definable morphisms $F_i$, where each $F_i$ only performs a change of variables in one valued field coordinate. One finishes by the chain rule for derivation. For a detailed argument of this kind we refer to \cite{CL}, Section 9.3 for $n=1$, and \cite{cln}, proof of Theorem 6.2.2 for $n>1$.
\end{proof}

Finally we formulate a general Fubini Theorem, in the Tonelli variant for non-negatively valued functions.

\begin{theorem}[Fubini-Tonelli]\label{FTon}
Let $\varphi$ be in $\cC_+(Z)$ for some $Z\subset X[n,m,r]$ and some $X$ in $\Def$.
Let $X[n,m,r]\to X[n-n',m-m',r-r']$ be a coordinate projection. Then $\varphi$ is $X$-integrable if and only if there exists a definable subassignment $Y$ of $Z$ whose complement in $Z$ has dimension $<n$ over $X$ such that, if we put $\varphi'=\11_Y\varphi$, then $\varphi'$ is $X[n-n',m-m',r-r']$-integrable and $\mu_{/X[n-n',m-m',r-r']}(\varphi')$ is $X$-integrable.
If this holds, then
$$
\mu_{/X} ( \mu_{/X[n-n',m-m',r-r']}(\varphi') ) = \mu_{/X}(\varphi)
$$
in $\cC_+(X)$.
\end{theorem}
\begin{proof}
For $n=n'=0$ this is clear by the definitions of $\mu_X$ and the fact that $\cT$ is split. For $n>0$ and $n'>0$ the essential case to prove is when $n=n'=1$, which is the same statement as Lemma-Definition \ref{generalint}.
\end{proof}

\section{Direct image formalism}\label{sec:imdir}

Let $\Lambda$ be in $\Def$. From now on, all objects will be over $\Lambda$, where we continue to use the notation $\star_{/\Lambda}$ instead of $\star\to\Lambda$ to denote that some object $\star$ is considered over $\Lambda$.

\par
Consider $X$ in $\Def_\Lambda$.
For each integer $d\geq 0$, let $\cC^{\leq d}_+(X_{/\Lambda})$ be the ideal of $\cC_+(X)$ generated by characteristic functions $\11_Z$ of $Z\subset X$ which have relative dimension $\leq d$ over $\Lambda$. Furthermore, we put $\cC^{\leq -1}_+(X_{/\Lambda})=\{0\}$.

\par
For $d\geq 0$, define $C_+^d(X_{/\Lambda})$ as the quotient of semi-groups $\cC^{\leq d}_+(X_{/\Lambda}) / \cC^{\leq d-1}_+(X_{/\Lambda})$; its nonzero elements can be seen as functions having support of dimension $d$ and which are defined almost everywhere, that is, up to definable subassignments of dimension $<d$.

\par
Finally, put
$$
C_+(X_{/\Lambda}) := \bigoplus_{d\geq 0} C_+^d(X_{/\Lambda}),
$$
which is actually a finite direct sum since $C_+^d(X_{/\Lambda})=\{0\}$ for $d$ larger than the relative dimension of $X$ over $\Lambda$.

We introduce a notion of isometries for definable subassignments. This is some work since also residue ring and integer variables play a role.
\begin{definition}[Isometries]\label{isom}
Consider $\overline \ZZ:=\ZZ\cup\{-\infty,+\infty\}$. Extend the natural order on $\ZZ $ to $\overline \ZZ$ so that $+\infty$ is the biggest element, and $-\infty$ the smallest.

Define $\overline \ord$ on $h[1,0,0]$ as the extension of $\ord$ by $\overline \ord(0)=+\infty$.
Define $\overline \ord$ on $h[0,m,r]$ by sending $0$ to $+\infty$ and everything else to $-\infty$.
Define $\overline \ord$ on $h[n,m,r]$ by sending $x=(x_i)_i$ to $\inf_i \overline\ord (x_i)$.

Call a definable isomorphism $f:Y\to Z$ between definable subassignments $Y$ and $Z$ an isometry if and only if
$$
\overline\ord (y - y') = \overline\ord (f_K(y) - f_K(y'))
$$
for all $\cT$-fields $K$ and all $y$ and $y'$ in $Y(K)$, where $y-y'=(y_i-y_i')_i$.
 In the relative setting, let $f:Y\to Z$ be a definable isomorphism over $\Lambda$. Call $f$ an isometry over $\Lambda$ if for all $\cT$-fields $K$ and for all $\lambda\in\Lambda(K)$, one has that $f_\lambda:Y_\lambda\to Z_\lambda$ is an isometry, where $Y_\lambda$ is the set of elements in $Y(K)$ that map to $\lambda$, and $f_\lambda$ is the restriction of $f_K$ to $Y_\lambda$.
\end{definition}

\begin{definition}[Adding parameters]\label{addpar}
Let $f:Y\to Z$ and $f':Y'\to Z'$ be  morphisms  in $\Def$ with $Y'\subset Y[0,m,r]$ and $Z'\subset Z[0,s,t]$ for some $m,r,s$, and $t$. Say that $f'$ is obtained from $f$ by adding parameters, if the natural projections
$
p:Y'\to Y
$
and
$
r: Z'\to Z
$
are definable isomorphisms and if moreover the composition $r \circ f'$ equals $f\circ p$.
\end{definition}
There exist many isometries by the following lemma.
\begin{lem}\label{lemisom}
Let $\Lambda$ be a definable subassignment, and let $X$ be a definable subassignment over $\Lambda$ of relative dimension $\leq d$ over $\Lambda$. Then there exists a definable morphism $f=f_2\circ f_1: Z\subset \Lambda[d,m,r]\to X$ over $\Lambda$ for some $m,r$, such that $f_2$ is an isometry over $\Lambda$, and $f_1$ is obtained from the identity function $X\to X$ by adding parameters.
\end{lem}
\begin{proof}
We may suppose that $X$ is a definable subassignment of $\Lambda[d+n,a,b]$ for some $n,a,b$.
The lemma follows by a finite recursion process which allows one to decrease $n$ by $1$, by Definition \ref{def:T} and model theoretic compactness. Namely, if $n>0$, by the Jacobian property, compactness, and by adding extra parameters (using $b$-minimality) to control inverse functions on the piece where the derivative has order $<0 $, one can reduce to the case that $X$ is a definable subassignment of dimension $\leq d$ of $\Lambda[d+n-1,a,b]$ for some $a$ and $b$ by using a piecewise, isometric coordinate projection. To this end note that, for a function $F$ as in Definition \ref{jacf} with moreover $\ord (F') \geq 0$, the projection of the graph $\Gamma_F$ of $F$ onto the first coordinate is isometric on  $\Gamma_F$.
\end{proof}

We will now define the integrable functions (over $\Lambda$) inside $C_+(X_{/\Lambda})$, denoted by ${\rm I}C_+(X_{/\Lambda})$, for any definable subassignment $X_{/\Lambda}$. The main idea here is that integrability conditions should not change under pull-backs along isometries and under maps obtained from the identity function by adding parameters.
 Consider $\varphi$ in $\cC^{\leq d}_+(X_{/\Lambda})$ and its image $\overline \varphi$ in $C^d_+(X_{/\Lambda})$ for some definable subassignment $X_{/\Lambda}$ over $\Lambda$.
  By Lemma \ref{lemisom}, we can write $X$ as a disjoint union of definable subassignments $X_1$, $X_2$ such that there exists a definable morphism $f=f_2\circ f_1: Z\subset \Lambda[d,m,r]\to X_2$ for some $m,r$, $\11_{X_1}\varphi = 0$, $f_2$ is an isometry over $\Lambda$, and $f_1$ is obtained from the identity function $X_2\to X_2$ by adding parameters.
 Call $\overline\varphi$ integrable if and only if $f^*(\varphi)$ is $\Lambda$-integrable as in Lemma-Definition \ref{generalint}. Note that this condition is independent of the choice of the $X_i$ and $f$, by the existence of common refinements.
This defines the grade $d$ part ${\rm I}C^d_+(X_{/\Lambda})$ of ${\rm I}C_+(X_{/\Lambda})$, and one sets
$$
{\rm I}C_+(X_{/\Lambda}) := \sum_{d\geq 0}{\rm I}C^d_+(X_{/\Lambda}).
$$

\par
The following theorem gives the existence and uniqueness of integration in the fibers relative over $\Lambda$ (in all relative dimensions over $\Lambda$), in the form of a direct image formalism, by associating to any morphim $f:Y\to Z$ in $\Def_\Lambda$ a morphism of semi-groups $f_!$ from ${\rm I}C_+(Y_{/\Lambda})$ to ${\rm I}C_+(Z_{/\Lambda})$. This association happens to be a functor and the map $f_!$ sends a function to its integral in the fibers relative over $\Lambda$ (in the correct relative dimensions over $\Lambda$). The underlying idea is that isometries, inclusions, and definable morphisms obtained by adding parameters from an identity map should yield a trivial $f_!$ coming from the inverse of the pullback $f^*$, and further there is a change of variables situation and a Fubini-Tonelli situation that should behave as in Section \ref{fp}.

\begin{theorem}\label{di}
There exists a unique functor from $\Def_\Lambda$ to the category of semi-groups, which sends an object $Z$ in $\Def_\Lambda$ to the semi-group ${\rm I}C_+(Z_{/\Lambda})$, and a definable morphism $f:Y\to Z$ to a semi-group homomorphism $f_!:{\rm I}C_+(Y_{/\Lambda})\to {\rm I}C_+(Z_{/\Lambda})$, such that, for $\varphi$ in ${\rm I}C^d_+(Y_{/\Lambda})$ and a representative $\varphi^0$ in $\cC^{\leq d}_+(Y_{/\Lambda})$ of $\varphi$ one has:

\begin{itemize}

\item[]{{\rm M1 (Basic maps):}}

If $f$ is either an isometry or is obtained from an identity map $C\to C$ for some $C$ in $\Def$ by adding parameters, then $f_!(\varphi)$ is the class in ${\rm I}C^d_+(Z_{/\Lambda})$ of
$(f^{-1})^*(\varphi^0).$

\item[]{{\rm M2 (Inclusions):}}

If $Y\subset Z$ and $f$ is the inclusion function, then  $f_!(\varphi)$ is the class in ${\rm I}C^d_+(Z_{/\Lambda})$ of the unique $\psi$ in $\cC^{\leq d}_+(Z_{/\Lambda})$ with $f^*(\psi)=\varphi^0$ and $\psi\11_Y=\psi$.

\item[]{{\rm M3 (Fubini-Tonelli):}}

If $f: Y = \Lambda[d,m,r]\to Z = \Lambda[d-d',m-m',r-r']$ is a coordinate projection, then $\varphi^0$ can be taken by Theorem \ref{FTon} such that it is $\Lambda[d-d',m-m',r-r']$-integrable and then $f_!(\varphi)$ is the class in $IC^{d-d'}_+(\Lambda[d-d',m-m',r-r'] _{/\Lambda})$ of
$$
\mu_{/\Lambda[d-d',m-m',r-r']}(\varphi^0).
$$

\item[]{{\rm M4 (Change of variables):}}

If $f$ is a definable isomorphism over $\Lambda[0,m,r]$ with $Y\subset \Lambda[d,m,r]$ and $Z\subset \Lambda[d,m,r]$ then $\varphi^0$ and a $\psi\in \cC^{\leq d}_+(Y_{/\Lambda})$ can be taken by Theorem \ref{cov} such that $\varphi^0= \psi\LL^{-\ordjac_{/\Lambda[0,m,r]} f}$, and then $f_!(\varphi)$ is the class in ${\rm I}C^d_+(Z_{/\Lambda})$ of $(f^{-1})^*(\psi)$.
\end{itemize}
\end{theorem}

\begin{proof}[Proof of Theorem \ref{di}]
Uniqueness is clear. Indeed, one can always cut into finitely many pieces to control the relative dimensions, and on such pieces there always exists a finite composition of morphisms as in the basic situations that factor the respective restrictions of $f$ (this uses Lemma \ref{lemisom}). Existence follows from the properties in the previous sections which yield that the calculations of the direct images do not depend on the way $f$ is written as a finite composition of morphisms as in the basic situations.
\end{proof}

 Theorem \ref{di} thus yields a functor from the category $\Def_{\Lambda}$ to the category with objects ${\rm I}C_+(Z_{/\Lambda})$ and with homomorphisms of semi-groups (or even of semi-modules over $\cC_+(\Lambda)$ ) as morphisms.
 This functor is an embedding (that is, injective on objects and on morphisms).
 The functoriality property $(g\circ f)_! = g_!\circ f_!$ is a flexible form of a Fubini Theorem.

Note that $C_+(X_{/\Lambda})$ is a graded $\cC_+(X)$-semi-module (but not so for  ${\rm I}C_+(X_{/\Lambda})$ which is just a graded $\cC_+(\Lambda)$-semi-module).
Using this module structure, we can formulate the following form of the projection formula.
\begin{prop}
For every morphism $f : Y \rightarrow Z$ in
$\Def_\Lambda$, and every $\alpha$ in $\cC_+ (Z)$
and $\beta$ in ${\rm I}C_+ (Y_{/\Lambda})$,
$\alpha f_! (\beta)$ belongs to ${\rm I}C_+ (Z_{/\Lambda})$
if and only if $f^*(\alpha) \beta$ is in ${\rm I}C_+ (Y_{/\Lambda})$.
If these conditions are verified, then
$f_! (f^*(\alpha) \beta) = \alpha f_! (\beta)$.
\end{prop}
\begin{proof}
Follows from (the proof of) Theorem \ref{di} and Proposition \ref{proj}.  Indeed, if one factors $f$ as in the proof Theorem \ref{di} into a finite composition of morphism as in the basic situations, then the projection formula holds compatibly at each factor in the composition by Proposition \ref{proj}.
\end{proof}
The analogy with the direct image formalism of Theorem 14.1.1 of \cite{CL} with $S=\Lambda$ is now complete. An important ingredient of the proof of Theorem \ref{di} is that general definable morphisms can be factored, at least piecewise, into definable morphisms of the specified simple types falling under M1 up to  M4.

\section{Motivic integration on varieties with volume forms}

We implement motivic integration on rigid and on algebraic varieties with volume forms in our framework, leading to natural change of variables formulas and Fubini statements. In the respective cases, we compare our motivic integrals with integrals of \cite{LSeb} on rigid varieties and with integrals from the survey paper \cite{NiSe-survey} on algebraic varieties equipped with gauge forms. 



\subsection{Motivic integration on rigid varieties}
Let
$R=\cO_K$ be a complete discrete valuation ring with fraction field $K$ and
perfect residue field
$k$.
Let $X$ be a smooth 
quasi-compact and separated rigid variety over $K$ of dimension $n$ and let $\omega$ be a differential form of degree $n$ on $X$.
Write $p$ for the characteristic of $k$, and $e$ for the ramification if $p>0$ and $e=0$ if $p=0$, so that $K$ is a $(0,p,e)$-field.
Let $\cT$ be the analytic theory $\cT_{K}$ in the language $\cL_{K}$ as in Examples \ref{as} and \ref{exts} of Section \ref{exs}.
Note that one can naturally consider $X(L)$ for any complete $\cT$-field $L$.
Thus, by Lemma \ref{compl} and using a finite covering by affinoids, we can consider $X$ as a definable subassignment and we may also consider definable subassignments $Y\subset X$, as well as definable morphisms on $X$, and so on. This is clearly independent of the choice of finite covering of $X$ by affinoids.
Likewise, we can consider $\cC_+(X)$ by using any finite covering by affinoids.
Let $\varphi$ be in $\cC_+(X)$.
By Lemma \ref{lemisom} and using a finite affinoid covering, $X$ admits a finite covering by disjoint definable subassignments
$U_i$, equipped with definable isomorphisms
$g_i : O_i \to U_i$ with $O_i$ definable subassignements
of some $h[n, m_i, n_i]$.
One writes the pullback by $g_i$ of the restriction of $\omega$ to $U_i$
as $f_i dx_1\wedge \ldots \wedge dx_n$ with
$f_i$  a definable morphism to $h[1,0,0]$ (the $f_i$ are well defined only up to definable subassignments of dimension $< n$). Briefly, for obtaining  $f_i$, one may use the notions of $L$-analycity and $L$-analytic forms of \cite{Bour} to find $f_{iL}$ for any complete $\cT$-field $L$ and combine with Lemma \ref{compl}; alternatively to using analicity, the Jacobian property can be exploited to yield the same $f_i$, again up to definable subassignments of dimension $< n$. Similarly, denote the pullback by $g_i$ of the restriction of $\varphi$ to $U_i$ by $\psi_i$.
If $\psi_i\LL^{-\ord f_i}$ is integrable on $O_i$ for each $i$, with the convention that $\LL^{-\ord(0)}=0$, then we call $\varphi$ integrable on $X$ for the motivic measure associated to $\omega$, and then we define $\int_X\varphi |\omega|$ in $\cC_+(\mathrm{point})$ as the finite sum
$$
\sum_i \mu(\psi_i \LL^{-\ord f_i}),
$$
for the motivic measure $\mu$ as in section \ref{genint}.
That this is well defined follows from Theorem \ref{cov}, the motivic change of variables formula.

We will now link these integral $\int_X \varphi \vert \omega \vert$ to the integrals as defined in \cite{LSeb} for the smooth case with gauge form.
Let $K_0 (\Var_k )$ be the Grothendieck ring of varieties over $k$, moded out by the extra relations $[X] = [f (X)]$, for $f : X \to Y$
radicial. Note that radicial means that for every algebraically closed field $\ell$ over $k$  the induced map $X(\ell)\to Y(\ell)$ is injective, and that $[f (X)]$ is an abbreviation for the class of the constructible set given as the image of $f$ by Chevalley's Theorem. (In the case that $p=0$, the extra relations are redundant.) In any case, this ring is isomorphic to the Grothendieck ring of
definable sets with coefficients from $k$ for the theory of algebraically closed fields in the language of rings.

Similarly as the morphism $\gamma$ as in Section 16.3  of
\cite{CL}, there is a canonical morphism
$\delta : \cC_+ (\mathrm{point}) \to K_0 (\Var_k ) \otimes \AA$. Indeed, we may note that $\cC_+ (\mathrm{point}) $ is isomorphic to $\AA_+\otimes_\ZZ \cQ_+(\mathrm{point})$, and for $x=\sum_{i=1}^r a_i\otimes b_i$ with $a_i\in \AA_+$ and $b_i\in \cQ_+ (\mathrm{point})$, we may set $\delta(x) = \sum_{i} a_i \otimes [b_i]$, where $[b_i]$ is the class in $K_0 (\Var_k )$ of the constructible set obtained from $b_i$, essentially, by elimination of quantifiers for the theory of algebraically closed fields in the language of rings (Chevalley's Theorem). Technically, there are two ways to specify precisely how $\delta$ is given by quantifier elimination, both yielding the same statements for the propositions below, and where the subtleties come  from the presence of higher order residue rings sorts in $\cL_{\rm high}$. In the first alternative, one notes that only the residue field sort is needed in both propositions below and in the formulas (\ref{formula}) and (\ref{formula'}), and no other residue ring sort is needed to state and prove the propositions below. The (partial) definition of $\delta$ for such objects is clear by Chevalley's Theorem. In the second alternative, one wants to describe $\delta$ also when higher order residue rings are involved. In the mixed characteristic case, this is done using the remark for perfect residue fields at the end of section \ref{Q+} (the perfectness condition is void for algebraically closed residue fields). In the equicharacteristic zero case, one proceeds by identifying any complete $\cT$-field with $k_K((\pi))$, where $k_K$ is the residue field and $\pi$ the uniformizer given via the $\cL$-structure. The residue ring $k_K[\pi_K]/(\pi_K)^n$ for $n>1$ is then naturally identified with $k_K^n$ by choosing the vector space basis consisting of $1,\ldots,\pi_K^{n-1}$, and definable sets are thus mapped to definable sets in the one-sorted ring language; this clearly preserves definable bijections and one is again done by Chevalley's Theorem.


Suppose now that $\omega$ is a gauge form.
Using a weak N{\'e}ron model $\mX$ of $X$, an integral
$\int^{LS}_X \vert \omega \vert$
in the localization of  $K_0 (\Var_k)$ with respect to the class of
the affine line is defined in \cite{LSeb}.
Hence we may consider the image of $\int^{LS}_X \vert
\omega \vert$
in the further localization
$K_0 (\Var_k ) \otimes \AA$.

\begin{prop}\label{compa}
Let $X$ be a smooth, quasi-compact and separated rigid variety over $K$ endowed with a gauge form
$\omega$.
Then, with the above notation,
$\delta (\int_X \vert \omega \vert)$
is equal to the image of
$\int^{LS}_X \vert \omega \vert$ in $K_0 (\Var_k ) \otimes \AA$.
\end{prop}
\begin{proof}
Let $\mX$ be a weak N\'eron model for
$X$. For every connected component $C$ of $\mX_k=\mX\times_R
 k$, we denote by $\ord_C\omega$ the order of $\omega$ along $C$.
 If $\varpi$ is a uniformizer in $R$, then $\ord_C\omega$ is the
 unique integer $n$ such that $\varpi^{-n}\omega$ extends to a
 generator of
 $\Omega^{\dim(X)}_{\mX/R}$ at the generic point of $C$.
By Proposition 4.3.1 of \cite{LSeb},
it is enough to prove that
\begin{equation}\label{formula}
\int_X|\omega|=\LL^{-\dim(X)}\, \sum_{C\in
 \pi_0(\mX_k)}[C]\LL^{-\ord_C \omega},
\end{equation}
where $\pi_0(\mX_k)$ denotes the set of connected components of $\mX_k$.
Let $L$ be
a complete $(0,p,e)$-field
containing $K$. We denote by $\cO_L$  and $k_L$ the corresponding
valuation ring and residue field.
Since $L / K$ is unramified,
the canonical map $\mX(\cO_L)\to X(L)$ is a bijection by the weak N\'eron model property if $L$ is algebraic over $K$, and by base change properties if $L$ is not algebraic over $K$.
We denote by
$\pi_L :  X  (L) \to \mX(\cO_L)$ the inverse of that map.
By composing $\pi_L$ with the reduction map
$\mX (\cO_L) \to \mX_k (k_L)$, one gets a map
$\theta_L : X (L) \to \mX_k (k_L)$.
The maps $\theta_L$ are induced by a definable morphism
$\theta : X \to \mX_k $.
Since $\mX$ is smooth over $R$, for every point $a$ of $X$
there exist an open affinoid neighbourhood $Z$ of $a$ in $\mX$ and an \'etale morphism of affinoids
$h : Z \to \AA^n_R$.
Since moreover $\omega$ is a gauge form, it follows that there exists a definable isomorphism
$\lambda : Y\subset \mX_k[n,0,0] \to X$ such that, for each complete $\cT$-field $L$, any fiber of the projection $Y(L)\to \mX_k(k_L)$ is a translate of the open unit ball $\cM_L^n$ by some element in $\cO_L^n$,
$\theta \circ \lambda$ is the projection on the first factor, and, for any connected component $C$ of $\mX_k$,
the restriction
of $\lambda^* (\omega)$ to $Y\cap (C[n,0,0])$  is of the form
$\varpi^{\ord_C \omega} u dx_1 \wedge \cdots \wedge dx_n$
with $x_1$, \dots, $x_n$ the standard coordinates on
$h [n, 0, 0]$ and $u$ a definable morphism on $\mX_k \times B$ with $\ord(u)$ constantly zero.
The equality (\ref{formula}) now follows by an application of the Fubini-Tonelli \ref{FTon} to integrate $\LL^{-\ord_C \omega}$ over $Y\cap C[n,0,0]$, which is calculated using a finite affine cover of $C$, and summing over $C\in \pi_0(\mX_k)$.
\end{proof}

\subsection{Motivic integration on varieties with volume forms}

On algebraic varieties with volume forms, one proceeds similarly; we give details for the convenience of the reader. Let $K_0$ be a subfield of a $(0,p,e)$-field for some $p,e$, and let $\cT$ and $\cL$ be as in example \ref{sas} of Section \ref{exs} with $K_0$ playing the role of $R_0$.
Let $X$ be an algebraic variety defined over $K_0$ of dimension $n$ and let $\omega$ be a differential form of degree $n$ on $X$ (also called a volume form).
Considering $X(L)$ for any $\cT$-field $L$, and  using finite coverings by affine subvarieties, we can again consider $X$ as a definable subassignment and likewise for $\cC_+(X)$ and so on.
Let $\varphi$ be in $\cC_+(X)$.
By Lemma \ref{lemisom} and using a finite covering by affine subvarieties, $X$ admits a finite covering by disjoint definable subassignments
$U_i$, equipped with definable isomorphisms
$g_i : O_i \to U_i$ with $O_i$ definable subassignements
of some $h[n, m_i, n_i]$.
One writes the pullback by $g_i$ of the restriction of $\omega$ to $U_i$
as $f_i dx_1\wedge \ldots \wedge dx_n$ with
$f_i$  a definable morphism to $h[1,0,0]$ (the $f_i$ are well defined only up to definable subassignments of dimension $< n$). Similarly, denote the pullback by $g_i$ of the restriction of $\varphi$ to $U_i$ by $\psi_i$.
If $\psi_i\LL^{-\ord f_i}$ is integrable on $O_i$ for each $i$, with the convention that $\LL^{-\ord(0)}=0$, then we call $\varphi$ integrable on $X$ for the motivic measure associated to $\omega$, and we define $\int_X\varphi |\omega|$ in $\cC_+(\mathrm{point})$ as the finite sum
$$
\sum_i \mu(\psi_i \LL^{-\ord f_i}),
$$
with $\mu$ as in section \ref{genint}.
That this is well defined follows again from Theorem \ref{cov}.

We will prove a change of variables result, a Fubini theorem, and  give a link to the integrals from \cite{LSeb}, \cite{NiSe-survey}.

\begin{prop}[Change of variables]\label{covforms}
Let $f:X\to Y$ be a morphism of varieties of dimension $n$, defined over $K_0$. Let $\varphi$ be in $\cC_+(Y)$ and let $\omega_Y$ be a volume form on $Y$. Suppose that there is a definable subassignment $Z\subset X$ such that the restriction of $f$ to $Z$ is a definable isomorphism from $Z$ onto $f(Z)$ and such that $\11_{f(Z)}\varphi = \varphi$. Then one has that $\varphi$ is integrable on $Y$ for the motivic measure associated to $\omega_Y$ if and only if $\11_Z f^{*}(\varphi)$ is integrable on $X$ for the motivic measure associated to $|f^*(\omega_Y)|$ and, in that case,
$$
\int_X \11_Z f^{*}(\varphi) |f^*(\omega_Y)| = \int_{Y} \varphi |(\omega_Y)|.
$$
\end{prop}
\begin{proof}
By working on affine charts, we may suppose that $X$ and $Y$ are affine. By Lemma \ref{lemisom}, the above definitions, and the chain rule for derivation, the Proposition reduces to Theorem \ref{cov}.
\end{proof}

\begin{prop}[Fubini]\label{Fubforms}
Let $f:X\to Y$ be a morphism of varieties defined over $K_0$. Suppose that $X$ is of dimension $n+d$ and $Y$ is dimension $n$. Let $\varphi$ be in $\cC_+(X)$, let $\omega_Y$ be a volume form on $Y$, and $\omega_X$ a volume form on $X$. Further let $\omega$ be a differential form of degree $d$ on $X$. Suppose that there is a definable subassignment $Z\subset X$ such that, for each $\cT$-field $K$ and each point $z=(z_0,K)$ in $Z$, one has, at the stalk at $z_0 \in X(K)$, that $f^{*}(\omega_Y)\wedge\omega = \omega_X$.
Suppose moreover that $\11_{Z}\varphi = \varphi$.
Then there exists $\psi$ in $\cC_+(Y)$ such that,
for each $\cT$-field $K$ and each point $y=(y_0,K)$ in $Y$,
$$
i_y^*(\psi) =   \int_{X_y} \varphi_{|X_y} |\omega_{|X_y}|
$$
and
$$
\int_Y \psi |\omega_Y| = \int_{X} \varphi | \omega_X |,
$$
where $X_y$ is the reduced subvariety defined over $K$ associated to $f^{-1}(y_0)$,  $\omega_{|X_y}$ the restriction of $\omega$ to $X_y$,
and $\varphi_{|X_y}$ the restriction of $\varphi$ to $X_y$ (using $\cL(K)$ instead of $\cL$ as in Section \ref{eval}).
\end{prop}
An explicit and natural definition of a function $\psi$ as in Proposition \ref{Fubforms} is in fact also possible and is implicit in the proof.
\begin{proof}[Proof of Proposition \ref{Fubforms}]
Using affine charts one reduces to the case that $X$ and $Y$ are affine.
Now the Proposition reduces to Theorem \ref{FTon} by applying Lemma \ref{lemisom} once to $X$ over $Y$ (that is, with $Y$ in the role of $\Lambda$), and subsequently to $Y$ itself.
\end{proof}

Let us now simply write $K$ for $K_0$ and $R$ for $\cO_K$.
For a smooth variety $X$ over $K$, a weak N\' eron model for $X$ is a smooth $R$-variety
$\mX$
endowed with an isomorphism $\mX_{K} \to X$
such that moreover the natural map
$$\mX(\cO_{K'}) \to X(K')$$ is a bijection for any finite unramified extension $K'$ of $K$, where $\mX_{K}$ is $\mX \times_{R} K$. See \cite{NiSe-survey} for a more detailed treatment of these and similar notions. With the framework of this paper, we finish by computing certain motivic integrals on weak N\'eron models, similar to what is done in \cite{LSeb} and \cite{NiSe-survey}.

\begin{prop}\label{compa2}
Let $X$ be a smooth variety over $K$ endowed with a gauge form
$\omega$. Assume that $X$ admits a weak N\'eron model $\mX$.
Then one has
\begin{equation}\label{formula'}
\int_X|\omega|=\LL^{-\dim(X)}\, \sum_{C\in
 \pi_0(\mX_k)}[C]\LL^{-\ord_C \omega},
\end{equation}
where $\pi_0(\mX_k)$ denotes the set of connected components of $\mX_k$ with $\mX_k=\mX\times_R
 k$, and $\ord_C\omega$ the order of $\omega$ along $C$.
\end{prop}
\begin{proof}
The same proof as for Proposition \ref{compa} goes through, using affine charts instead of affinoids.
\end{proof}

\bibliographystyle{amsplain}

\end{document}